\documentclass{amsart}

\newcommand{\pictures}{}
\usepackage{color}

\ifdefined\pictures
\usepackage[pdf]{pstricks}
\usepackage{pst-node}
\usepackage{pst-plot}
\fi

\usepackage[
	pdftitle={Hard edge},
	pdfauthor={Leonid Chekhov and Paul Norbury},
	ocgcolorlinks,
	linkcolor=linkblue,
	citecolor=linkred,
	urlcolor=linkred]
{hyperref}

\usepackage{amsfonts}
\usepackage{amssymb}
\usepackage{bbm}
\usepackage{bm}
\usepackage[hang,flushmargin]{footmisc}
\usepackage[left=25mm,right=25mm,top=25mm,bottom=25mm]{geometry}
\usepackage{mathpazo}
\def\appendix#1{
\addtocounter{section}{1} \setcounter{equation}{0}
\renewcommand{\thesection}{\Alph{section}}
\section*{Appendix \thesection\protect\indent\quad
#1}
}

\usepackage{mathtools}

\usepackage{xcolor}
     \definecolor{linkred}{rgb}{0.6,0,0}
     \definecolor{linkblue}{rgb}{0,0,0.6}

\theoremstyle{plain}
     \newtheorem{theorem}{Theorem}
     
     \newtheorem{proposition}{Proposition}[section]
     
     \newtheorem{corollary}[proposition]{Corollary}

\theoremstyle{definition}
     \newtheorem{example}[proposition]{Example}
     \newtheorem{definition}[proposition]{Definition}
     \newtheorem{remark}[proposition]{Remark}

\newcommand{\res}{\mathop{\mathrm{Res}}}

\newcommand{\pp}{\boldsymbol{p}}
\newcommand{\bc}{\mathbb{C}}
\newcommand{\bn}{\mathbb{N}}
\newcommand{\bp}{\mathbb{P}}
\newcommand{\br}{\mathbb{R}}
\newcommand{\bz}{\mathbb{Z}}

\newcommand{\ch}{\mathcal{H}}
\newcommand{\cl}{\mathcal{L}}
\newcommand{\modm}{\mathcal{M}}

\newcommand{\cp}{\mathcal{P}}

\newcommand {\dd}{\mathrm{d}}

\newcommand {\zz}{\bm{z}}

\newcommand{\un}{1\!\!1}
\def\res{\mathop{{\rm Res}}}

\newcommand{\tcr}{\textcolor{red}}
\newcommand{\tcb}{\textcolor{blue}}

\def\tr{{\mathrm{tr\,}}}

\newcommand*{\Cdot}{\raisebox{-0.5ex}{\scalebox{1.8}{$\cdot$}}} %{\raisebox{-0.25ex}{\scalebox{1.2}{$\cdot$}}}
\numberwithin{equation}{section}

\begin{document}

\title{Topological recursion with hard edges}
\author{Leonid Chekhov \and Paul Norbury}
\address{Steklov Mathematical Institute and Lab. Poncelet, Moscow, Russia; Niels Bohr Institute, Copenhagen, Denmark}
\address{School of Mathematics and Statistics, University of Melbourne, VIC 3010 Australia}
\email{\href{mailto:chekhov@mi.ras.ru}{chekhov@mi.ras.ru}, \href{mailto:pnorbury@ms.unimelb.edu.au}{pnorbury@ms.unimelb.edu.au}}
\thanks{}
\date{\today}

\begin{abstract}
We prove a Givental type decomposition for partition functions that arise out of topological recursion applied to spectral curves.  Copies of the Konstevich-Witten KdV tau function arise out of regular spectral curves and copies of the Brezin-Gross-Witten KdV tau function arise out of irregular spectral curves.  We present the example of this decomposition for the matrix model with two hard edges and spectral curve $(x^2-4)y^2=1$.
\end{abstract}

\maketitle

\setlength{\parskip}{0pt}
\tableofcontents
\setlength{\parskip}{6pt}

\section{Introduction}  \label{sec:intro}

The partition functions for Gromov-Witten invariants of $\bp^1$, the Gaussian Hermitian matrix model, the Legendre ensemble, and enumeration of dessins d'enfant, which are formal series in an infinite sequence of variables $\{\hbar,v^{k,\alpha}\mid k\in\bn,\alpha\in\{1,2\}\}$, have in common a decomposition given by a differential operator $\hat{R}$ acting on the product of two species of the Kontsevich-Witten KdV tau function $Z^{\text{KW}}$ or of the Brezin-Gross-Witten KdV tau function $Z^{\text{BGW}}$---defined in Section~\ref{sec:BGW}.
\begin{align}  \label{examples}
Z^{\text{GW}}(\hbar,\{v^{k,\alpha}\})%=\exp\sum_g\hbar^{1-g}\left\langle\exp\left\{\sum_{k\geq 0}^{\infty}\psi^k_iev_i^\ast(v_\alpha)t^{\alpha}_k\right\}\right\rangle^g
&=\hat{R}\cdot \hat{T}_1\cdot Z^{\text{KW}}(2\hbar,\{\sqrt{2}v^{k,1}\})Z^{\text{KW}}(-2\hbar,\{i\sqrt{2}v^{k,2}\}) \nonumber\\
Z^{\text{GUE}}(\hbar,\{v^{k,\alpha}\})%=\int_{H_N}\exp{V(M)}DM
&=\hat{R}\cdot \hat{T}_2\cdot Z^{\text{KW}}(2\hbar,\{\sqrt{2}v^{k,1}\})Z^{\text{KW}}(-2\hbar,\{i\sqrt{2}v^{k,2}\})\\
Z^{\text{Leg}}(\hbar,\{v^{k,\alpha}\})%=\int_{H_N[-2,2]}\exp{V(M)}DM
&=\hat{R}\cdot \hat{T}_3\cdot Z^{\text{BGW}}(2\hbar,\{\sqrt{2}v^{k,1}\})Z^{\text{BGW}}(-2\hbar,\{i\sqrt{2}v^{k,2}\}) \nonumber\\
Z^{\text{Des}}(\hbar,\{v^{k,\alpha}\})%=\int_{H_N[-2,2]}\exp{V(M)}DM
&=\hat{R}\cdot \hat{T}_4\cdot Z^{\text{BGW}}(-\frac12\hbar,\{\frac{i}{\sqrt{2}}v^{k,1}\})Z^{\text{KW}}(32\hbar,\{4\sqrt{2}v^{k,2}\}) \nonumber.
\end{align}
The operator $\hat{R}$, which is the exponential of a quadratic differential operator, is common to all four models, whereas $\hat{T}_i$ are operators of translations $v^{k,\alpha}\mapsto v^{k,\alpha}+c^{k,\alpha}_i$ described in Section~\ref{sec:trans}.
%$$Z^{\text{GW}}= \exp\sum_g\hbar^{2g-2} \sum_dq^d\left\langle\exp\left\{\sum_{i\geq 0}^{\infty}\tau_i(\omega)t_i+\tau_i(1)s_i\right\}\right\rangle^g_d\int _{[\overline{\modm}_{g,n}(\bp^1,d)]^{vir}}\prod_{i=1}^n \psi_i ^{b_i} ev_i^*(\alpha_i),$$
The partition function $Z^{\text{GW}}$ stores ancestor Gromov-Witten invariants of $\bp^1$.  Its decomposition in \eqref{examples} is a particular case of Givental's decomposition of partition functions of Gromov-Witten invariants of targets $X$ with semi-simple quantum cohomology \cite{GivGro} which applies more generally to partition functions of semi-simple cohomological field theories.  It is usually expressed as a function of variables corresponding to cohomology classes in $H^*(\bp^1)$ denoted $\{t^{k,\beta}\}$ which are related to the variables in the decomposition by $v^{k,1}=\frac{1}{\sqrt{2}}(t^{k,1}+t^{k,2})$, $v^{k,2}=\frac{i}{\sqrt{2}}(-t^{k,1}+t^{k,2})$.  The partition functions $Z^{\text{GUE}}$ and $Z^{\text{Leg}}$ store moments of the probability measure $\int_{H_N}\exp{(NV(M))}DM$ as asymptotic expansions in $\hbar=1/N^2$ for $N\to\infty$  where $H_N$ consists of $N\times N$ Hermitian matrices.  For $Z^{\text{GUE}}$, we use $V(M)=-\tr(M^2)$.  For $Z^{\text{Des}}$ we use $V(M)=-\tr M\cdot\chi_{[0,\infty)}(M)$ where the indicator function is defined by $\chi_{[0,\infty)}(M):=\prod_{i=1}^N\chi_{[0,\infty)}(\lambda_i)$ and for $A\subset\br$, $\chi_A(x)=1$ %=\left\{\begin{array}{ll}1&x\in A\\0&x\notin A\end{array}\right.$ 
for $x\in A$ and 0 otherwise.  This produces a potential with infinite wall, or {\em hard edge}, at zero eigenvalue.  For $Z^{\text{Leg}}$ we use $V(M)=\chi_{[-2,2]}(M)$ (again $\chi_{[-2,2]}(M):=\prod_{i=1}^N\chi_{[-2,2]}(\lambda_i)$) i.e. one restricts eigenvalues to lie in the interval $[-2,2]$ and sets $V(M)=0$.  This produces a potential with infinite walls, or hard edges, at eigenvalues $-2$ and $2$.  Note that for the integration over the compact domain $[-2,2]^N$ we do not need a Gaussian term for convergence.   The decomposition of $Z^{\text{GUE}}$ is a decomposition of the partition function for a Hermitian matrix model with Gaussian potential  proven by the first author in \cite{Ch95} and in fact gives an example of Givental's decomposition via an associated cohomological field theory \cite{ACNP}.  The decomposition of $Z^{\text{Leg}}$ into two copies of $Z^{\text{BGW}}$ is described in detail in this paper as a particular example of the more general result involving copies of both $Z^{\text{KW}}$ and $Z^{\text{BGW}}$.  An example of mixed $Z^{\text{BGW}}$ and $Z^{\text{KW}}$ factors is given by $Z^{\text{Des}}$ the partition function for enumeration of dessins d'enfant \cite{DNoTop}.  The conclusion is that the partition functions $Z^{\text{KW}}$ and $Z^{\text{BGW}}$ are fundamental to a large class of partition functions arising from many areas.

For the two choices of $V(M)$ above, the limit
$$y(x)=\lim_{N\to\infty}N^{-1}\int_{H_N}\left\langle\tr\frac{1}{x-M}\right\rangle\exp{(NV(M))}DM
$$
is a holomorphic function near $x=\infty$ which analytically continues to define a Riemann surface, known as a spectral curve, given as a double cover of the $x$-plane $x=z+1/z$ on which $y(x)dx$ extends to a well-defined differential $r(z)dz$ for $r(z)$ a rational function.  One can also associate a Riemann surface to Gromov-Witten invariants of
$\bp^1$ via an associated Landau-Ginzburg model.  Each of the examples in \eqref{examples} can be formulated in terms of a recursive construction of meromorphic differentials defined on the associated Riemann surface known as {\em topological recursion}.  In this paper we prove a decomposition theorem for partition functions arising out of topological recursion that generalises the examples in \eqref{examples}.

Topological recursion developed by Eynard, Orantin and the first author \cite{CEyHer,CEO,EOrInv} produces invariants $\omega_{g,n}$ for integers $g\geq 0$ and $n\geq 1$, which we will refer to as {\em correlators}, of a Riemann surface $\Sigma$ equipped with two meromorphic functions $x, y: \Sigma\to \mathbb{C}$ and a bidifferential $B(p_1,p_2)$ for $p_1, p_2 \in \Sigma$.  The zeros $\cp_\alpha$ of $dx$ must be simple and disjoint from the zeros of $dy$.     We refer to the data $S=(\Sigma,B,x,y)$ as a {\em spectral curve}.   We allow $\Sigma$ to be (a possibly disconnected) open subset of a compact Riemann surface,  in which case $S$ is known as a {\em local spectral curve}.   For integers $g \geq 0$ and $n \geq 1$, the correlator $\omega_{g,n}$ is a multidifferential on $\Sigma$ or, in other words, a tensor product of meromorphic differentials on $\Sigma^n$.  It is defined recursively via
$$\omega_{0,1}(p)=-y(p)dx(p),\quad \omega_{0,2}(p_1,p_2)=B(p_1,p_2)
$$
which are used to define the kernel in a neighbourhood of $p_2=\cp_\alpha$ for $dx(\cp_\alpha)=0$
$$ K(p_1,p_2)=\frac{1}{2}\frac{\int_{\sigma_\alpha(p_2)}^{p_2}B(p,p_1)}{(y(p_2)-y(\sigma_\alpha(p_2)))dx(p_2)}.
$$
The point $\sigma_\alpha(p)\in \Sigma$ is defined to be the unique point $\sigma_\alpha(p)\neq p$ close to $\alpha$ such that $x(\sigma_\alpha(p))=x(p)$ which is well-defined since each zero $\cp_\alpha$ of $dx$ is assumed to be simple.  For $L=\{2,...,n\}$ define
\begin{equation}  \label{TRrec}
\omega_{g,n}(p_1,\pp_{L})=\sum_{\alpha=1}^D\res_{p=\cp_\alpha}K(p_1,p) \bigg[\omega_{g-1,n+1}(p,\sigma_\alpha(p),\pp_{L})+ \mathop{\sum_{g_1+g_2=g}}_{I\sqcup J=L}^\circ \omega_{g_1,|I|+1}(p,\pp_I) \, \omega_{g_2,|J|+1}(\sigma_\alpha(p),\pp_J) \bigg]
\end{equation}
where the outer summation is over the zeros $\cp_\alpha$ of $dx$ and the $\circ$ over the inner summation means that we exclude terms that involve $\omega_1^0$.    The recursive definition of $\omega_{g,n}(p_1, \ldots, p_n)$ uses only local information around zeros of $\dd x$ so a local spectral curve containing the zeros of $dx$ is sufficient.  A zero of $dx$ is {\em regular} if $y$ is analytic there.   A spectral curve is regular if $y$ is analytic at all zeros of $dx$.  In this paper we consider {\em irregular} spectral curves where $y$ may have a simple pole at any zero of $dx$.  The correlators $\omega_{g,n}$ are polynomial in a basis of differentials $v^{k,m}=V^{k,m}(p_i)$ on $\Sigma$ depending only $x$ and $B$---defined in \eqref{Vdiff} in Section~\ref{sec:partfun}.  Define the topological recursion partition function of the spectral curve $S=(\Sigma,B,x,y)$ by
$$Z^S(\hbar,\{v^{k,i}\})=\exp\left(\sum_{g,n}\frac{\hbar^{g-1}}{n!}\omega_{g,n}(\{v^{k,i}\})\right).
$$

The topological recursion partition function of the curve $x=\frac{1}{2}y^2$ (equipped with the Cauchy kernel $B=dy_1dy_2/(y_1-y_2)^2$) is the Kontsevich-Witten KdV tau function $Z=Z^{\text{KW}}$ and we write $\omega^{\text{Airy}}_{g,n}$ for the correlators of this curve known as the Airy curve due to its relation to the differential equation satisfied by the Airy function.  Similarly, the topological recursion partition function of the curve $xy^2=\frac{1}{2}$ yields the Brezin-Gross-Witten KdV tau function $Z=Z^{\text{BGW}}$, defined in Section~\ref{sec:BGW}, and we write $\omega^{\text{Bes}}_{g,n}$ due to a relation of the curve with the Bessel equation \cite{DNoTop1}. The BGW model was first identified with the KdV $\tau$-function in \cite{MMS}.

For a general spectral curve $S$, it is straightforward to prove that the asymptotic behaviour, or largest order principal part, of $\omega_{g,n}$ near each zero $\cp_\alpha$ of $dx$ is given by the correlators for the local model of the curve $x=\frac{1}{2}y^2$ and $xy^2=\frac{1}{2}$, i.e.  $\omega_{g,n}\sim c^{2g-2+n}\omega^{\text{Airy}}_{g,n}$ for some constant $c\in\bc$ near any regular zero of $dx$ and $\omega_{g,n}\sim c^{2g-2+n}\omega^{\text{Bes}}_{g,n}$ near any irregular zero of $dx$ where $y$ has a simple pole.  What is much deeper is that $\omega_{g,n}$ can be constructed completely from copies of $\omega^{\text{Airy}}_{g,n}$ and $\omega^{\text{Bes}}_{g,n}$.  This is described in terms of the partition functions in Theorem~\ref{th:main} below.

In \cite{DOSSIde} Dunin-Barkowski, Orantin, Shadrin and Spitz proved that the partition function $Z^S$ of a regular spectral curve satisfying a finiteness assumption possesses a decomposition in terms of products of $Z^{\text{KW}}$ acted on by differential operators built out of spectral curve data.  Furthermore, they showed that this decomposition coincides with a decomposition of Givental \cite{GivGro} for partition functions $Z$ arising out of semi-simple cohomological field theories.  An immediate consequence is that, under some assumptions on the spectral curve, topological recursion produces partition functions $Z$ for semi-simple cohomological field theories.

The results of \cite{DOSSIde} require the spectral curve to have regular singularities, i.e. $dy$ must be analytic at the zeros of $dx$.  The main result of this paper is a generalisation of the decomposition theorem to allow irregular singularities.
\begin{theorem}  \label{th:main}
Consider a spectral curve $S=(\Sigma,B,x,y)$ with $m$ irregular zeros of $dx$ at which $y$ has simple poles, and $D-m$ regular zeros.  There exist operators $\hat{R}$, $\hat{T}$ and $\hat{\Delta}$ defined in Definition~\ref{ops} determined explicitly by $(\Sigma,B,x,y)$ such that
\begin{equation}  \label{eq:main}
Z^S=\hat{R}\hat{T}\hat{\Delta}Z^{\text{BGW}}(\hbar,\{v^{k,1}\})\cdots Z^{\text{BGW}}(\hbar,\{v^{k,m}\})Z^{\text{KW}}(\hbar,\{v^{k,m+1}\})\cdots Z^{\text{KW}}(\hbar,\{v^{k,D}\}).
\end{equation}
Moreover, $\hat{R}$ depends only on $(\Sigma,B,x)$, $\hat{T}$ is a translation operator, and $\hat{\Delta}$ acts by rescaling $\hbar$ and $v^{k,i}$.
\end{theorem}
\begin{remark} Theorem~\ref{th:main} generalises the result of \cite{DOSSIde} even when applied to a regular spectral curve since it contains a translation term $\hat{T}$ that does not appear in the decomposition theorem of \cite{DOSSIde}.  More precisely, \cite{DOSSIde} considers a restricted class of spectral curves on which the function $y$ is determined by $(\Sigma,B,x)$ together with the $m$ numbers $dy(\cp_\alpha)\in\bc$ and corresponds to cohomological field theories with flat identity.  In particular \cite{DOSSIde} does not apply to Weil-Petersson volumes studied by Mirzakhani \cite{MirWei} which gives a fundamental example of a cohomological field theory without flat identity.  This is despite the proof in \cite{EynRec,EOrWei} that Weil-Petersson volumes do indeed arise out of topological recursion applied to a spectral curve.  Theorem~\ref{th:main} remedies this situation.  The translation term in Theorem~\ref{th:main} (for regular spectral curves) corresponds to a translation of cohomological field theories---see \cite{PPZRel}.
\end{remark}
\begin{remark}
In the examples \eqref{examples} the different operators $\hat{\Delta}$ are visible via the differential rescalings of $\hbar$.  The operator $\hat{R}$ is built out of the data $(\Sigma,B,x)$---see Section~\ref{sec:decomp}.  It is the same for all four examples of \eqref{examples} reflecting the fact that the associated spectral curves differ only in the definition of $y$, i.e. $(\Sigma,B,x)$ is the same in these four cases.
\end{remark}
\begin{remark}
The partition functions arising from Gromov-Witten invariants, cohomological field theories, matrix models and topological recursion possess two sets of natural coordinates---{\em flat} coordinates and {\em canonical} coordinates.  Theorem~
\ref{th:main} is expressed with respect to canonical coordinates.  The change of coordinates between canonical and flat is given by an $m\times m$ matrix, i.e. it is linear and independent of the first parameter $d$.  This change of coordinates appears in Givental's decomposition as a further operator %$\hat{\Psi}$
acting on the left of \eqref{eq:main}.
\end{remark}
%Note Theorem~\ref{th:main} requires no assumptions on the spectral curve even in the regular case.  In \cite{DOSSIde} there is an assumption on the spectral curve---a relationship between $y$ and $B$---but their proof of the decomposition goes through without this assumption.  They used the assumption to deduce particularly nice properties of the corresponding cohomological field theory.

The proof of the formula \eqref{eq:main} is built up progressively via special cases proven throughout the text.  The reader may find some of these simpler versions more digestable.  The basic cases of $(\hat{R},\hat{T},\hat{\Delta})=(Id,Id,Id)$ and $D=1$ appear in \eqref{airypart} and \eqref{bespart}.  The case of $(\hat{R},\hat{T},\hat{\Delta})=(Id,Id,\hat{\Delta})$ which gives rise to a topological field theory appears in \eqref{ztop}.  Translations can be best understood via the case $(\hat{R},\hat{T},\hat{\Delta})=(Id,\hat{T},\hat{\Delta})$ and $D=1$ given in \eqref{eq:trans}.

In Section~\ref{sec:BGW} we recall the definitions of the two KdV tau functions which form the fundamental pieces of the decomposition.   In Section~\ref{sec:partfun} we introduce the decomposition without the differential operator $\hat{R}$ via the elementary topological part of the correlators.  In Section~\ref{sec:decomp} we prove the decomposition \eqref{eq:main}.  We apply the decomposition to the example of the Legendre ensemble in Section~\ref{sec:legendre} and demonstrate an application of the decomposition \eqref{eq:main} pictorially.  In this paper we use the convention $\bn=\{0,1,2,...\}$.

\subsection{Acknowledgements}
The authors would like to thank Maxim Kazarian, Nicolas Orantin, Sergey Shad\-rin and Peter Zograf for useful conversations and the referee for many useful comments.  PN was partially supported by the Australian Research Council grant DP170102028. The work of L.Ch. was partially supported by the Russian Foundation for Basic Research (Grant No. 17-01-00477) and by the ERC Advanced Grant 291092 ''Exploring the Quantum Universe'' (EQU).

\section{Kontsevich-Witten and Brezin-Gross-Witten tau functions}   \label{sec:BGW}

The fundamental components $Z^{\text{KW}}(\hbar,t_0,t_1,...)$ and $Z^{\text{BGW}}(\hbar,t_0,t_1,...)$ of the factorisation \eqref{eq:main} are tau functions of the KdV hierarchy.  The Kontsevich-Witten tau function $Z^{\text{KW}}$ was introduced in \cite{WitTwo}, and the Brezin-Gross-Witten tau function $Z^{\text{BGW}}$ arises out of a unitary matrix mode studied in \cite{BGrExt,GWiPos}.  A tau function $Z(t_0,t_1,...)$ of the KdV hierarchy (equivalently the KP hierarchy in odd times $p_{2k+1}=t_k/(2k+1)!!$) gives rise to a solution of the KdV hierarchy via $Z=\exp{F}$, $U=\hbar\frac{\partial^2 F}{\partial t_0^2}$
\begin{equation}\label{kdv}
U_{t_1}=UU_{t_0}+\frac{\hbar}{12}U_{t_0t_0t_0},\quad U(t_0,0,0,...)=f(t_0).
\end{equation}
The first equation in the hierarchy is the KdV equation \eqref{kdv}, and later equations $U_{t_k}=P_k(U,U_{t_0},U_{t_0t_0},...)$ for $k>1$ determine $U$ uniquely from $U(t_0,0,0,...)$.

The Kontsevich-Witten tau function $Z^{\text{KW}}$ is defined by $U(t_0,0,0,...)=t_0$ (and is in fact determined uniquely by \eqref{kdv} and the string equation $F_{t_0}=\frac{1}{2}t_0^2+\sum t_{i+1}F_{t_i}$ i.e. the higher equations giving $U_{t_k}$ for $k>1$ are automatically satisfied).  It is famously a generating function for intersection numbers over the moduli space of stable curves equipped with tautological line bundles $L_i\to\overline{\modm}_{g,n}$, $i=1,...,n$ and $\psi_i=c_1(L_i)$.

\begin{theorem}[Witten-Kontsevich 1992 \cite{KonInt,WitTwo}]
$$F^{\text{KW}}(\hbar,t_0,t_1,...)=\sum_{g,n}\hbar^{g-1}\frac{1}{n!}\sum_{\vec{k}\in\bn^n}\int_{\overline{\modm}_{g,n}}\prod_{i=1}^n\psi_i^{k_i}t_{k_i}
$$
%$$ F^{\text{KW}}(\hbar,t_0,t_1,...)=\sum_g\hbar^{2g}\langle\tau_0^{k_0}\tau_1^{k_1}...\rangle_g\frac{t_0^{k_0}}{k_0!}\frac{t_1^{k_1}}{k_1!}...$$
\end{theorem}
Its first few terms are given by
$$F^{KW}(\hbar,t_0,t_1,...)=\hbar^{-1}(\frac{t_0^3}{3!}+\frac{t_0^3t_1}{3!}+\frac{t_0^4t_2}{4!}+...)+\frac{t_1}{24}+...
$$
Its dispersionless limit is nontrivial: $\displaystyle\lim_{\hbar\to 0}U=\frac{t_0}{1-t_1}+\frac{t_0^2t_2}{2(1-t_1)^3}+...$.

It arises via topological recursion applied to the Airy curve \cite{EOrTop}.  For
$$
S_{\text{Airy}}=\{x=\frac{1}{2}z^2,\ y=z,\ B=\frac{dzdz'}{(z-z')^2}\}
$$
and $\omega_{g,n}^{\text{Airy}}$ defined by \eqref{TRrec}, we have
$$\omega_{g,n}^{\text{Airy}}=\sum_{\vec{k}\in\bz_+^n}\int_{\overline{\modm}_{g,n}}\prod_{i=1}^n\psi_i^{k_i}(2k_i+1)!!\frac{dz_i}{z_i^{2k_i+2}}.
$$
Out of the correlators $\omega_{g,n}^{\text{Airy}}$, \eqref{eq:part} builds a partition function which in this case using $V^k(z)=(2k+1)!!\frac{dz}{z^{2k+2}}$ from Definition~\ref{auxdif} gives
\begin{equation}\label{airypart}
Z^{S_{\text{Airy}}}=Z^{\text{KW}}.
\end{equation}
This is the first case of \eqref{eq:main} where $D=1$ and $\hat{R}=Id=\hat{T}=\hat{\Delta}$.

%\subsection{Brezin-Gross-Witten tau function}

The Brezin-Gross-Witten solution of the KdV hierarchy is defined by the initial condition
$$
U(t_0,0,0,...)=\frac{\hbar}{8(1-t_0)^2}.
$$
The first few terms of its tau function are given by
$$ \log Z^{\text{BGW}}=F^{\text{BGW}}(\hbar,t_0,t_1,...)=\frac{1}{8}t_0+\frac{1}{16}t_0^2+\frac{1}{24}t_0^3+...+\hbar(\frac{3}{128}t_1+\frac{9}{128}t_0t_1+...)+..
$$
and we see that its dispersionless limit is trivial:
$$\lim_{\hbar\to 0}U=0.$$

It arises via topological recursion applied to the Bessel curve \cite{DNoTop1}
$$
S_{\text{Bes}}=\{x=\frac{1}{2}z^2,\ y=\frac{1}{z},\ B=\frac{dzdz'}{(z-z')^2}\}
$$
as follows.  Write
$$
\omega_{g,n}^{\text{Bes}}=\sum_{\mu\in\bz_+^n}b_{g,n}(\mu_1, \ldots, \mu_n)\prod_{i=1}^n\frac{dz_i}{z_i^{\mu_i+1}}
$$
for the correlators of topological recursion applied to the curve $S_{\text{Bes}}$.  As above $V^k(z)=(2k+1)!!\frac{dz}{z^{2k+2}}$ and it is proven in \cite{DNoTop1} that
$$F_g^{S_{\text{Bes}}}=F^{\text{BGW}}_g=\sum_n\frac{1}{n!}\sum_{k\in\bn^n}\frac{b_{g,n}(2k_1+1, \ldots, 2k_n+1)}{\prod_{i=1}^n(2k_i+1)!!}t_{k_1}\dots t_{k_n}$$
hence
\begin{equation}\label{bespart}
Z^{S_{\text{Bes}}}=Z^{\text{BGW}}
\end{equation}
which is again a case of \eqref{eq:main} for $D=1$ and $\hat{R}=Id=\hat{T}=\hat{\Delta}$.

%$$F^{\text{BGW}}_g=\sum_n\frac{1}{n!}\sum_{\mu\in\bz_+^n}\frac{b_{g,n}(\mu_1, \ldots, \mu_n)}{\prod_{i=1}^n\mu_i}s^{|\mu|}p_{\mu_1}\dots p_{\mu_n},\quad Z^{\text{BGW}}=\exp{\sum_g \hbar^{2g}F^{\text{BGW}}_g}.$$

\begin{remark}
Kontsevich and Soibelman \cite{KSoAir} have studied topological recursion via an algebraic structure which they call an {\em Airy structure} referring to the fact that a regular spectral curve locally resembles the Airy curve.  The more general setup of irregular spectral curves that locally resemble the Bessel curve also determines an Airy structure, %also fits into the picture of Kontsevich and Soibelman
using the technique of abstract topological recursion \cite{ABCO}.
\end{remark}
%Define $U_g(\mu_1, \ldots, \mu_n)$ via
%\[ F_{g,n}=\sum U_g(\mu_1, \ldots, \mu_n)\prod_{i=1}^n z_i^{-\mu_i}\]
%where $d_1...d_nF_{g,n}=\omega_{g,n}$.  Then $F_{g,n}=F_{g,n}(p_1,p_3,...,p_{2g-1})$ where $p_k=z_i^{-k}$ and
%$$F^{BGW}_g(p_1,p_3,...)=\sum_n\frac{1}{n!}F_{g,n}.$$

%$$\omega_{g,n}^B=\sum_{\vec{\alpha}\geq 0}\prod_{i=1}^n\frac{(2\alpha_i+1)!!dz_i}{z^{2\alpha_i+2}}\langle\prod_{j=1}^n\tau_{\alpha_j}\rangle_{g,n}$$

\section{Partition function for topological recursion}  \label{sec:partfun}

In this section we define the partition function $Z^S$ built out of the correlators $\omega_{g,n}$ of the spectral curve $S$.  It is a generating function for all $\omega_{g,n}$ with the substitution of variables $v^{k,i}$ for differentials on the curve. We then give a leisurely introduction to the formula \eqref{eq:main} by considering only part of this formula, obtained by arranging $\hat{R}=Id=\hat{T}$.

Any correlator $\omega_{g,n}(p_1,...,p_n)$ has the property that its principal part in any $p_i$ at any zero $\cp_i$ of $dx$ is skew-invariant under the local involution defined by $dx$ around $\cp_i$.  Eynard \cite{EynInv} defined a collection of auxiliary differentials (defined below) on the curve which span all those meromorphic differentials with principal part at any zero $\cp$ of $dx$ skew-invariant under the local involution defined by $dx$ around $\cp$. Hence $\omega_{g,n}$ is a polynomial in these auxiliary differentials.
\begin{definition}\label{auxdif}
For a Riemann surface $\Sigma$ equipped with a meromorphic function $x:\Sigma\to\bc$ define the auxiliary differentials on $\Sigma$ as follows:
\begin{equation}  \label{Vdiff}
V^\alpha_0(p)=B(\cp_\alpha,p),\quad V^\alpha_{k+1}(p)=d\left(\frac{V^\alpha_k(p)}{dx(p)}\right),\ \alpha=1,...,D,\quad k=0,1,2,...
\end{equation}
where $B(\cp_\alpha,p)$ is evaluation at $\cp_\alpha$, a zero of $dx$ .  Evaluation of any meromorphic differential $\omega$ at a simple zero $\cp$ of $dx$ is defined by
$$
\omega(\cp):=\res_{p=\cp}\frac{\omega(p)}{\sqrt{2(x(p)-x(\cp))}}%=\left.\frac{\omega(p)}{ds}\right|_{s=0}.
$$
where we choose a branch of $\sqrt{x(p)-x(\cp)}$ once and for all at each $\cp$ to remove the $\pm1$ ambiguity.  
\end{definition}
The meromorphic differentials $V^\alpha_k(p)$ defined above constitute local Krichever--Whitham systems \cite{Kr} of 1-differentials skew-invariant with respect to local involutions.  In \cite{EynInv} Eynard defines $d\xi_{\alpha,k}(p)$, using local coordinates, which give the principal part of $V^\alpha_k(p)$ and serve a similar role.  

For each locally defined involution $\sigma_\alpha$ defined in a neighbourhood of a zero $\cp_\alpha$ of $dx$, we have $V^\alpha_k(p)+V^\alpha_k(\sigma_\alpha(p))$ is analytic at $\cp_\alpha$.  The $V^\alpha_k(p)$ form a basis for meromorphic differentials which have principal part skew-invariant under each involution $\sigma_\alpha$ and $\omega_{g,n}$ is a polynomial in them:
$$\omega_{g,n}(p_1,...,p_n)=\sum_{\vec{\alpha},\vec{k}}c_{g,\vec{\alpha},\vec{k}}\prod_{i=1}^nV_{k_i}^{\alpha_i}(p_i).
$$
The partition function of a spectral curve $S=(\Sigma,B,x,y)$ is defined by:
\begin{equation}  \label{eq:part}
Z^S(\hbar,\{v^{k,\alpha}\})=\exp\left(\sum_{g,n}\hbar^{g-1}\frac{1}{n!}\omega_{g,n}(p_1,...,p_n)|_{\{V_{k_i}^{\alpha_i}(p_i)=v^{k_i,\alpha_i}\}}\right).
\end{equation}

\subsection{Topological field theory and asymptotic behaviour of $\omega_{g,n}$.}

Given a spectral curve $S=(\Sigma,B,x,y)$, around any regular zero $\cp_i$ of $dx$ the pair $(\Sigma,x)$ resembles the Airy curve $x=\frac{1}{2}s^2$.  A consequence of this proven in \cite{EOrInv} is that near a branch point, the asymptotic behaviour of $\omega_{g,n}(p_1,...,p_n)$ is described by $\omega^{\text{Airy}}_{g,n}(s_1,...,s_n)$.   More precisely, consider a local variable $s$ in a neighbourhood of $a_i$ chosen so that $x=x(\cp_i)+\frac{1}{2}\epsilon^2s^2$, for $\epsilon>0$ a small real constant. With respect to this local coordinate $y=y(\cp_i)+\eta_i^{1/2}\epsilon s+...$ where
\begin{equation} \label{etareg}
\eta_i=dy(\cp_i)^2=\res_{z=\cp_i}\frac{(dy)^2}{dx}.
\end{equation}
The dominant asymptotic term as $\epsilon\rightarrow 0$ is
\begin{equation} \label{eq:asym}
\omega_{g,n}(p_1,...,p_n)=\epsilon^{6-6g-3n}dy(\cp_i)^{2-2g-n}\omega_{g,n}^{\text{Airy}}(s_1,...,s_n)+O(\epsilon^{7-6g-3n})
\end{equation}
where $s_i=s(p_i)$ is the local coordinate $s$ evaluated at the point $p_i\in\Sigma$.%, $\omega_{g,n}^{\text{Airy}}(s_1,...,s_n)$ is homogeneous of degree $6-6g-3n$ in the $s_i$, and $O(s^{7-6g-3n})$ means a homogeneous Laurent polynomial in $s_i$ of homogeneous degree $7-6g-3n$.

The analogous behaviour near an irregular zero $\cp_i$ of $dx$ where the spectral curve resembles the Bessel curve $xy^2=\frac{1}{2}$ also holds.  With respect to the local variable $s$ defined above %such that $x=x(\cp_i)+\frac{1}{2}s^2$ 
we now have $y=\eta_i^{1/2}s^{-1}+...$ for
\begin{equation} \label{etairreg}
\eta_i=(ydx)(\cp_i)^2=\res_{z=\cp_i}y^2dx.
\end{equation}
The dominant asymptotic term as $\epsilon\rightarrow 0$ is %need to discuss a local multidifferential
\begin{equation} \label{eq:asymirr}
\omega_{g,n}(p_1,...,p_n)=\epsilon^{2-2g-n}(ydx)(\cp_i)^{2-2g-n}\omega_{g,n}^{\text{Bessel}}(s_1,...,s_n)+O(\epsilon^{3-2g-n}).
\end{equation}
%where $\omega_{g,n}^{\text{Bessel}}(s_1,...,s_n)$ is homogeneous of degree $-2g-2n$ in the $s_i$, and $O(s^{1-2g-2n})$ means a homogeneous Laurent polynomial in $s_i$ of homogeneous degree $1-2g-2n$.

Collect the top order pole parts of $\omega_{g,n}(p_1,...,p_n)$ at each $\cp_i$, which have order $6g-6+4n$ by \eqref{eq:asym}, respectively $2g-2+2n$ by \eqref{eq:asymirr}, into a single correlator $\omega_{g,n}^{\text{top}}(p_1,...,p_n)$.  Alternatively define $\omega_{g,n}^{\text{top}}(p_1,...,p_n)$ using a local spectral curve $S_0$ built out of $S$ as follows.  Given a spectral curve $S=(\Sigma,B,x,y)$, define the local spectral curve $S_0=(\Sigma,B_0,x,y)$ with $B_0$ the trivial Bergman kernel---it is given in a local coordinate $s$ around any zero $\cp_i$ of $dx$ defined by $x=x(\cp_i)+\frac{1}{2}s^2$ by $B_0(p,p')=\frac{ds(p)ds(p')}{(s(p)-s(p'))^2}$.  The correlators $\omega_{g,n}^{\text{top}}(p_1,...,p_n)$ of $S_0$ consist of the top order pole parts of $\omega_{g,n}(p_1,...,p_n)$ at each zero of $dx$.

In preparation for studying the full formula \eqref{eq:main}, we restate \eqref{eq:asym} and \eqref{eq:asymirr} by \eqref{ztop} below in terms of the partition functions $Z^{\text{KW}}$ and $Z^{\text{BGW}}$ by replacing the $\hat{R}$ operator in \eqref{eq:main} with the identity operator.

Recall that a two-dimensional topological field theory (2D TFT) is a vector space $H$ and a sequence of symmetric linear maps
\[ I_{g,n}:H^{\otimes n}\to \bc\]
for integers $g\geq 0$ and $n>0$ satisfying the following conditions.  The map $I_{0,2}$ is a non-degenerate bilinear form on $H$ i.e. it defines a metric $h$, with dual bivector $\Delta=h^{\alpha\beta}e_{\alpha}\otimes e_{\beta}$ (defined with respect to a basis $\{ e_{\alpha}\}$ of $H$).  The map $I_{0,3}$ together with $h=I_{0,2}$ defines a
product $\Cdot$ on $H$ via
$$h(v_1\Cdot v_2,v_3)=I_{0,3}(v_1,v_2,v_3)$$
with identity $\un$ given by the dual of $I_{0,1}=\un^*=h(\un,\cdot)$.  It satisfies the natural insertion of $\un$ condition $I_{g,n+1}(\un\otimes v_1\otimes...\otimes v_n)=I_{g,n}(v_1\otimes...\otimes v_n)$ and gluing conditions 
$$
I_{g,n}(v_1\otimes...\otimes v_n)=I_{g-1,n+2}(v_1\otimes...\otimes v_n\otimes\Delta)=I_{g_1,|I|+1}\otimes I_{g_2,|J|+1}\big(\bigotimes_{i\in I}v_i\otimes\Delta\otimes\bigotimes_{j\in J}v_j\big)$$
for $g=g_1+g_2$ and $I\sqcup J=\{1,...,n\}$.
 
Via the natural isomorphism $H^0({\overline{\modm}_{g,n})}\cong\bc$ we consider $I_{g,n}:H^{\otimes n}\to H^0({\overline{\modm}_{g,n}})$ and integrate these classes next to Chern classes of the tautological line bundles $\cl_i$ on the moduli space of stable curves $\overline{\modm}_{g,n}$.  With respect to a basis $\{ e_{\nu_i}\}$ of $H$, we define the partition function:
\begin{equation}  \label{partfn}
Z(\hbar,\{v^{k,j}\})=\exp\sum_{g,n}\hbar^{g-1}\frac{1}{n!}\int_{\overline{\modm}_{g,n}}I_{g,n}(e_{\nu_1},...,e_{\nu_n})\cdot\prod_{j=1}^nc_1(\cl_j)^{k_j}v^{k_j,\nu_j}.
\end{equation}
The partition function \eqref{partfn} for the trivial dimension 1 TFT, where $I_{g,n}(\un^{\otimes n})=1$, stores intersection numbers of $\psi$ classes on $\overline{\modm}_{g,n}$ and hence $Z(\hbar,\{v^{k,1}\})=Z^{\text KW}(\hbar,\{v^{k,1}\})$, the Kontsevich-Witten partition function, as described in Section~\ref{sec:BGW}.

A 2D TFT is known as {\em semisimple} if its associated Frobenius algebra $(H,\eta,\Cdot)$ is {\em semisimple}, i.e.
$$H\cong\bc\oplus\bc\oplus...\oplus\bc,\quad\langle e_i,e_j \rangle=\delta_{ij}\eta_i,\quad e_i\cdot e_j=\delta_{ij}e_i$$
for some $\eta_i\in\bc \setminus \{0\}$, $i=1,...,D$ where $e_i^{(j)}=\delta_{ij}$ is the standard basis.  For a semisimple 2D TFT, the $I_{g,n}$ decompose into a sum of 1-dimensional TFTs and hence are extremely simple.  A 1-dimensional TFT depends on a single complex number $I_{0,1}(\un)=\eta\in\bc$ which determines $I_{g,n}(\un^{\otimes n})=\eta^{1-g}$.  Its partition function \eqref{partfn} is simply $Z^{\text KW}(\eta^{-1}\hbar,\{u^{k,1}\})$ since the $\eta^{1-g}$ is naturally absorbed by the $\hbar^{g-1}$.  The coordinate $u^{k,1}$ corresponds to the unit vector $\un$, and we instead use $v^{k,1}=\eta^{\frac12}u^{k,1}$ corresponding to an orthonormal basis, so $Z^{\text KW}(\eta^{-1}\hbar,\{u^{k,1}\})=Z^{\text KW}(\eta^{-1}\hbar,\{\eta^{-\frac12}v^{k,1}\})$.  Hence the partition function of a semisimple 2D TFT, where $I_{g,n}$ vanishes on mixed monomials in the $e_i$ and $I_{g,n}(e_i^{\otimes n})=\eta_i^{1-g}$, is given by a product
$$Z(\hbar,\{v^{k,j}\})=Z^{\text{KW}}(\eta_1^{-1}\hbar,\{\eta_1^{-\frac12}v^{k,1}\})\cdots Z^{\text{KW}}(\eta_D^{-1}\hbar,\{\eta_D^{-\frac12}v^{k,D}\}).
$$
For storing the numbers $I_{g,n}(e_{\nu_1},...,e_{\nu_n})$ this may seem more complicated than necessary, however it brings useful insight to the formula \eqref{eq:main}, and is precisely necessary to store the highest order parts $\omega_{g,n}^{\text{top}}(p_1,...,p_n)$ of the correlators $\omega_{g,n}$ of a spectral curve $S=(\Sigma,B,x,y)$.  A generalisation of this idea couples a TFT to $Z^{\text BGW}(\hbar,\{v^{k,1}\})$ via $Z^{\text BGW}(\eta^{-1}\hbar,\{\eta^{-\frac12}u^{k,1}\})$.    Thus
\begin{equation}  \label{ztop}
Z^{S_0}(\hbar,\{v^{k,i}\})=\hat{\Delta}Z^{\text{BGW}}(\hbar,\{v^{k,1}\})\cdots Z^{\text{BGW}}(\hbar,\{v^{k,m}\})Z^{\text{KW}}(\hbar,\{v^{k,m+1}\})\cdots Z^{\text{KW}}(\hbar,\{v^{k,D}\})
\end{equation}
where $S$, hence $S_0$, has $k$ irregular zeros of $dx$ at which $y$ has simple poles, and $D-m$ regular zeros, $\hat{\Delta}(\hbar)=\eta_i^{-1}\hbar$ in the $i$th factor and $\hat{\Delta}(v^{k,i})=\eta_i^{-\frac12}v^{k,i}$.

The expressions \eqref{eq:asym} and \eqref{eq:asymirr}, hence also \eqref{ztop}, depend only on the local behaviour of $x$ and $y$ in a neighbourhood of $\cp_i$, and are independent of $B$.  The formula \eqref{eq:main} strengthens this result to show that all lower asymptotic terms of $\omega_{g,n}$, hence $\omega_{g,n}$ itself, can also be obtained from $\omega_{g,n}^{\text{Airy}}$ and $\omega_{g,n}^{\text{Bessel}}$ by also using $B$.

\section{Givental decomposition}  \label{sec:decomp}

\subsection{Translations}  \label{sec:trans}
The translation term $\hat{T}$ in \eqref{eq:main} translates the arguments $v^{k,\alpha}$ in the tau functions:
$$Z^{\text{KW}}(\hbar,\{v^{k,\alpha}\})\mapsto Z^{\text{KW}}(\hbar,\{v^{k,\alpha}+c^{k,\alpha}\}),\quad Z^{\text{BGW}}(\hbar,\{v^{k,\beta}\})\mapsto Z^{\text{BGW}}(\hbar,\{v^{k,\beta}+c^{k,\beta}\}).
$$
Translations arise from local expansions of $y$ near each zero $\cp_\alpha$ of $dx$.
To study these translations, we restrict to the case where $dx$ has a single zero.  Let $x=\frac{1}{2}z^2$, $y=\hspace{-2mm}\displaystyle\sum_{k=-1}^{\infty}y_kz^{k}$.  This is a deformation of the Bessel curve, or a deformation of the Airy curve when $y_{-1}=0$ and $y_1\neq 0$.  General deformations of the Airy curve were studied in \cite{DOSSIde,EynRec}.  Propositions~\ref{th:defair} and \ref{th:defbes} below show that the correlators of deformations of the Airy and Bessel curves depend linearly on the coefficients $a_{g,n}$ and $b_{g,n}$ of the correlators $\omega_{g,n}^{\text{Airy}}$ and $\omega_{g,n}^{\text{Bes}}$ defined by
\begin{equation}  \label{eq:abcoeff}
\omega_{g,n}^{\text{Airy}}=\sum_{\mu\in\bz_+^n}a_{g,n}(\mu)\prod_{i=1}^n\frac{dz_i}{z_i^{\mu_i+1}},\quad \omega_{g,n}^{\text{Bes}}=\sum_{\mu\in\bz_+^n}b_{g,n}(\mu)\prod_{i=1}^n\frac{dz_i}{z_i^{\mu_i+1}}
\end{equation}
and homogeneously in the coefficients $y_k$ of $y$. %(which corresponds to $y_1=1$, $y_k=0$, $k>1$).
We will begin with the statement of the known result for deformations of the Airy curve.
\begin{proposition}[\cite{DOSSIde,EynInv,EynRec}]   \label{th:defair}
For $x=\frac{1}{2}z^2$, $y=\displaystyle\sum_{k=1}^{\infty}y_kz^k$, $B=\frac{dzdz'}{(z-z')^2}$ and $2g-2+n>0$,
\begin{equation}  \label{defair}
\omega_{g,n}(z_1,...,z_n)=y_{1}^{2-2g-n}\sum_{m=0}^{\infty}\sum_{\vec{d},\vec{\alpha}}\prod_{i=1}^n\frac{dz_i}{z_i^{d_i+1}}\frac{(-1)^m}{m!}a_{g,n+m}(\vec{d},\vec{\alpha})\prod_{k=1}^m\left(\frac{y_{\alpha_k-2}}{y_{1}}\frac{1}{\alpha_k}\right)
\end{equation}
where $\vec{d}=(d_1,...,d_n)$, $\vec{\alpha}=(\alpha_1,...,\alpha_m)$ with $\alpha_k>3$, $a_{g,n}$ is defined in \eqref{eq:abcoeff} and the product $\displaystyle\prod_{k=1}^m(\cdot)$ is 1 when $m=0$.  Note that $a_{g,n+m}(\vec{d},\vec{\alpha})= 0$ when any $d_i$ or $\alpha_i$ is even or when $|\vec{d}|+|\vec{\alpha}|\neq 6g-6+3n$ so the sum is finite.
\end{proposition}
The statements of Proposition~\ref{th:defair} in the literature, such as Lemma 3.5 in \cite{DOSSIde}, use intersection numbers on the moduli space of stable curves related via $a_{g,n}(2m_1+1,...,2m_n+1)=\int_{\overline{\modm}_{g,n}}\prod_{k=1}^n(2m_k+1)!!\psi_k^{m_k}$. 

\begin{proposition}   \label{th:defbes}
For $x=\frac{1}{2}z^2$, $y=\displaystyle\sum_{k=-1}^{\infty}y_kz^k$, $B=\frac{dzdz'}{(z-z')^2}$ and $2g-2+n>0$,
\begin{equation}  \label{defbes}
\omega_{g,n}(z_1,...,z_n)=y_{-1}^{2-2g-n}\sum_{m=0}^{g-1}\sum_{\vec{d},\vec{\alpha}}\prod_{i=1}^n\frac{dz_i}{z_i^{d_i+1}}\frac{(-1)^m}{m!}b_{g,n+m}(\vec{d},\vec{\alpha})\prod_{k=1}^m\left(\frac{y_{\alpha_k-2}}{y_{-1}}\frac{1}{\alpha_k}\right)
\end{equation}
%$$\omega_{g,n}=y_{1}^{2-2g-n}\sum_{m=0}^{g-1}\mathop{\sum_{\alpha\in\bz_+^m}}_{d\in\bz_+^n}\prod_{i=1}^n\frac{dz_i}{z_i^{2d_i+2}}\frac{(-1)^m}{m!}\prod_{k=1}^my_{\alpha_k}b_{g,n+m}(\vec{d},\vec{\alpha})$$
where $\vec{d}=(d_1,...,d_n)$, $\vec{\alpha}=(\alpha_1,...,\alpha_m)$ with $\alpha_k>1$, $b_{g,n}$ is defined in \eqref{eq:abcoeff} and the product $\displaystyle\prod_{k=1}^m(\cdot)$ is 1 when $m=0$.  Note that $b_{g,n+m}(\vec{d},\vec{\alpha})= 0$ when any $d_i$ or $\alpha_i$ is even or when $|\vec{d}|+|\vec{\alpha}|\neq 2g-2+n$ so the sum is finite.
%$$\omega_{g,n}(z_1,...,z_n)=y_{-1}^{2-2g-n}\sum_{m=0}^{\infty}\frac{(-1)^m}{m!}\sum_{\vec{\alpha}}\prod_{k=1}^m\frac{(2\alpha_k+1)!!y_{2\alpha_k+1}}{y_{-1}}\sum_{\vec{d}}\prod_{i=1}^n\frac{(2d_i+1)!!dz_i}{z^{2d_i+2}}\langle\prod_{j=1}^n\tau_{d_j}\prod_{k=1}^m\tau_{\alpha_k+1}\rangle_{g,n+m}$$
\end{proposition}
\begin{proof}
The dependence of the kernel $K(z_1,z)$ on $y_{2k-1}$ in the recursion \eqref{TRrec} is given by
\begin{align*}
K(z_1,z)&=\frac{1}{2}\frac{\int_{-z}^zB(z',z_1)}{(y(z)-y(-z))dx(z)}=\frac{dz_1}{2(z_1^2-z^2)\displaystyle{\sum_{k=0}^{\infty}} y_{2k-1}z^{2k-1}dz}\\
&=\frac{1}{y_{-1}}\frac{zdz_1}{2(z_1^2-z^2)(1+\displaystyle{\sum_{k=1}^{\infty}} \frac{y_{2k-1}}{y_{-1}}z^{2k})dz}
=\frac{1}{y_{-1}}\frac{zdz_1}{2(z_1^2-z^2)dz}\sum_{m=0}^{\infty}(-1)^m\left( \sum_{k=1}^{\infty}\frac{y_{2k-1}}{y_{-1}}z^{2k}\right)^m.
\end{align*}
Thus each summand contributes a term of homogeneous degree $m$ in the $y_{2k-1}$ for $k\geq 1$ to $K(z_1,z)$.  Thus we can write $\omega_{g,n}$ as
\begin{align*}
\omega_{g,n}(z_1,z_2,..,z_n)&=\res_{z=0}K(z_1,z) \bigg[\omega_{g-1,n+1}(z,-z,z_2,..,z_n)+\hspace{-5mm} \mathop{\sum_{g_1+g_2=g}}_{I\sqcup J=\{2,..,n\}}^{\circ} \hspace{-3mm}\omega_{g_1,|I|+1}(z,\zz_I) \, \omega_{g_2,|J|+1}(-z,\zz_J) \bigg]\\
&=y_{-1}^{2-2g-n}\sum_{m=0}^{g-1}\sum_{\vec{d},\vec{\alpha}}\prod_{i=1}^n\frac{dz_i}{z_i^{d_i+1}}(-1)^m
\prod_{k=1}^m\frac{y_{\alpha_k}}{y_{-1}}C_{g,n+m}(\vec{d},\vec{\alpha})
\end{align*}
for some constants $C_{g,n+m}(\vec{d},\vec{\alpha})$ which we will show coincide with coefficients $\frac{1}{m!}b_{g,n+m}(\vec{d},\vec{\alpha}+2)$ of $\omega_{g,n}^{\text{Bes}}$ (for $\vec{\alpha}+2=(\alpha_1+2,...,\alpha_k+2)$).  The dependence on only odd $\alpha_k$ is realised by the same property of $b_{g,n+m}(\vec{d},\vec{\alpha})$.  It is easy to prove by induction that $\omega_{g,n}$ is a polynomial in the $y_k$ for odd positive $k$.

A variational formula for a family of spectral curves depending on a parameter was proven in \cite{CMMV,EOrInv}.
When the variation of $ydx$ is given by integration of $B(p_1,p_2)$ around a generalised cycle
$$\frac{\partial}{\partial y_k}(-ydx)=-z^{k+1}dz=-\frac{1}{k+2}d(z^{k+2})=-\res_{z_0=z}\frac{1}{k+2} z_0^{k+2}B(z_0,z)=\res_{z_0=\infty} \frac{1}{k+2}z_0^{k+2}B(z_0,z)
$$
it determines the variation of $\omega_{g,n}$ to be integration of $\omega_{g,n+1}$ around the same generalised cycle
$$\frac{\partial}{\partial y_k}\omega_{g,n}(z_1,...,z_n)=\res_{z_0=\infty} \frac{1}{k+2}z_0^{k+2}\omega_{g,n+1}(z_0,z_1,...,z_n).
$$
Write $[y_iy_j...y_k]f$ for the coefficient of $y_iy_j...y_k$ in a polynomial $f$ of the variables $y_i$.
\begin{align*}
\left[\prod_{k=1}^my_{\alpha_k}\right]\omega_{g,n}&=\frac{1}{|\text{Aut}(\alpha)|}\left.\frac{\partial^m}{\partial_{y_{\alpha_1}}...\partial_{y_{\alpha_m}}}\omega_{g,n}\right|_{y_i=0}\\
&=\frac{1}{|\text{Aut}(\alpha)|}\left.\res_{z_{n+1}=\infty}...\res_{z_{n+m}=\infty}\prod_{k=1}^m\frac{1}{\alpha_k+2}z_{n+k}^{\alpha_k+2}\omega_{g,n+m}(z_1,...,z_{n+m})\right|_{y_i=0}\\
&=\frac{1}{|\text{Aut}(\alpha)|}\res_{z_{n+1}=\infty}...\res_{z_{n+m}=\infty}\prod_{k=1}^m\frac{1}{\alpha_k+2}z_{n+k}^{\alpha_k+2}\omega_{g,n+m}^{\text{Bes}}(z_1,...,z_{n+m})\\
&=\sum_{\vec{d}}\prod_{i=1}^n\frac{dz_i}{z_i^{d_i+1}}\frac{(-1)^m}{|\text{Aut}(\alpha)|}\prod_{k=1}^m\frac{1}{\alpha_k+2}b_{g,n}(d_1,...,d_n,\alpha_1+2,...,\alpha_m+2)
\end{align*}
since $\displaystyle\omega_{g,n}^{\text{Bes}}=\sum_{\mu\in\bz_+^n}b_{g,n}(\mu)\prod_{i=1}^n\frac{dz_i}{z_i^{\mu_i+1}}$.  Here $|\text{Aut}(\alpha)|=d_1!...d_j!$ is the order of the automorphism group of $\alpha=(\alpha_1,...,\alpha_m)=(\overbrace{\alpha_1,..,\alpha_1}^{d_1},...,\overbrace{\alpha_j,..,\alpha_j}^{d_j})$ which contains $j\leq m$ distinct $\alpha_k$.  After summing over all $m$-tuples, we replace $1/|\text{Aut}(\alpha)|$ with $1/m!$, and we shift each $\alpha_k$  to $\alpha_k-2$ to get \eqref{defbes} as required.
\end{proof}
It may be helpful to read \eqref{defbes} via small genus examples:
\begin{align*}
\omega_{1,n}&= y_{-1}^{-n}\prod_{i=1}^n\frac{dz_i}{z_i^2}\cdot\frac{1}{8}(n-1)!=y_{-1}^{-n}\prod_{i=1}^n\frac{dz_i}{z_i^2}\cdot b_{1,n}(1,1,...,1)
\\
\omega_{2,n}&=y_{-1}^{-n-2}\prod_{i=1}^n\frac{dz_i}{z_i^2}\left(\sum_{j=1}^n\frac{1}{z_j^2}\frac{9}{256}(n+1)!-\frac{y_1}{y_{-1}}\frac{3}{256}(n+2)! \right)\\
&=y_{-1}^{-n-2}\prod_{i=1}^n\frac{dz_i}{z_i^2}\left(\sum_{j=1}^n\frac{1}{z_j^2}b_{2,n}(3,1,...,1)-\frac{y_1}{y_{-1}}\frac{1}{3}b_{2,n+1}(1,...,1,3)\right)
\\
\omega_{3,n}&=y_{-1}^{-n-4}\prod_{i=1}^n\frac{dz_i}{z_i^2}\Big(\sum_{j=1}^n\frac{1}{z_j^4}\frac{75}{8192}(n+3)!+\sum_{i,j}\frac{1}{z_i^2z_j^2}\frac{189}{20480}(n+3)!\\
&\quad-\left[\sum_j\frac{1}{z_j^2}\frac{y_1}{y_{-1}}\frac{63}{20480}(n+4)!
+\frac{y_3}{y_{-1}}\frac{15}{8192}(n+4)!\right]+\frac{y_1^2}{y_{-1}^2}\frac{21}{40960}(n+5)!
\Big)\\
&=y_{-1}^{-n-4}\prod_{i=1}^n\frac{dz_i}{z_i^2}\Big(\sum_{j=1}^n\frac{1}{z_j^4}b_{3,n}(5,1,...,1)+\sum_{i,j}\frac{1}{z_i^2z_j^2}b_{3,n}(3,3,1,...,1)\\
&\ \ -\left[\sum_j\frac{1}{z_j^2}\frac{y_1}{y_{-1}}\frac{1}{3}b_{3,n+1}(3,1,...,1,3)
+\frac{y_3}{y_{-1}}\frac{1}{5}b_{3,n+1}(1,1,...,1,5)\right]+\frac{(-1)^2}{2!}\frac{y_1^2}{y_{-1}^2}\frac{1}{3^2}b_{3,n+2}(1,...,1,3,3)
\Big)
\end{align*}
Note that we place the arguments of the symmetric function $b_{g,n}$ at different ends to emphasise their origin.

\subsection{Graphical expansion}  \label{sec:graph}
In this section we generalise the weighted graphical expansion of the correlators $\omega_{g,n}$ for regular spectral curves proven in \cite{DOSSIde,EynInv} to allow for irregular spectral curves.

Given a set $\{1,...,.D\}$ which will correspond to the zeros of $dx$ on a spectral curve consider the following set of decorated graphs.
\begin{definition}
For a graph $\gamma$ denote by
$$V(\gamma),\quad E(\gamma),\quad H(\gamma),\quad L(\gamma)=L^*(\gamma)\sqcup L^\bullet(\gamma)$$
its set of vertices, edges, half-edges and leaves.  The disjoint splitting of $L(\gamma)$ into ordinary leaves, $L^*$, and dilaton leaves, $L^\bullet$, is part of the structure on $\gamma$.  The set of half-edges consists of leaves and oriented edges so there is an injective map $L(\gamma)\to H(\gamma)$ and a multiply-defined map $E(\gamma)\to H(\gamma)$ denoted by $E(\gamma)\ni e\mapsto \{e^+,e^-\}\subset H(\gamma)$.
The map sending a half-edge to its vertex is given by $v:H(\gamma)\to V(\gamma)$.  Decorate $\gamma$ by  functions:
\begin{align*}
g&:V(\gamma)\to\bn\\
\alpha&:V(\gamma)\to\{1,...,D\}\\
p&:L^*(\gamma)\stackrel{\cong}{\to}\{p_1,p_2,...,p_n\}\subset\Sigma\\
k&:H(\gamma)\to\bn
\end{align*}
such that  $k|_{L^\bullet(\gamma)}>1$ and $n=|L^*(\gamma)|$.  We write $g_v=g(v)$, $\alpha_v=\alpha(v)$, $\alpha_\ell=\alpha(v(\ell))$, $p_\ell=p(\ell)$, $k_\ell=k(\ell)$.
The {\em genus} of $\gamma$ is $g(\gamma)=\displaystyle b_1(\gamma)+\hspace{-2mm}\sum_{v\in V(\gamma)}\hspace{-2mm}g(v)$ and $\gamma$ is stable if any vertex labeled by $g=0$ is of valency $\geq 3$.   We write $n_v$ for the valency of the vertex $v$.
For $2g-2+n$, define $\Gamma_{g,n}$ to be the set of all stable connected genus $g$ decorated graphs with $n$ ordinary leaves.
\end{definition}
%In the following we choose an orientation on each edge of $\gamma$ and produce a weighted sum which is independent of this choice.  An orientation defines maps
%$$ v_1,v_2:E(\gamma)\to V(\gamma),\quad  h_1,h_2:E(\gamma)\to H(\gamma)$$
%that map edges to the first and second vertices and corresponding half-edges.

We now express the correlators of a spectral curve as a sum over decorated graphs, with vertices weighted by coefficients of $\omega_{g,n}^{\text{Airy}}$ or $\omega_{g,n}^{\text{Bes}}$, edges weighted by coefficients of local expansions of $B$, ordinary leaves weighted by differentials determined by $(\Sigma,B,x)$, and dilaton leaves weighted by coefficients of local expansions of $y$. We follow the exposition in \cite{DOSSIde} for the regular case and generalise it to the irregular case.

A spectral curve $S=(\Sigma,B,x,y)$ defines a disjoint splitting $V(\gamma)=V^{\text{reg}}(\gamma)\sqcup V^{\text{irreg}}(\gamma)$ of $V(\gamma)$ into regular and irregular vertices.  Label the zeros of $dx$ by $\cp_1,...,\cp_D$ and define a vertex $v$ of $\gamma$ to be irregular if $y$ has a pole at $\cp_{\alpha_v}(v)$.  The spectral curve $S$ also defines weights on any decorated graph $\gamma\in\Gamma_{g,n}$ which will be used to produce correlators $\omega_{g,n}$ of $S$ as weighted sums over all decorations on graphs of type $(g,n)$.  The weights are defined as follows.
\begin{definition}  \label{weights}
{\em Vertex weights.}
$$W(v)=\left\{\begin{array}{ll}\hbar^{g-1}y_{1,\alpha_v}^{2-2g_v-n_v}a_{g_v,n_v}(\{k_h\mid h\in H(v)\}),&v\in V^{\text{reg}}\\
\hbar^{g-1}y_{-1,\alpha_v}^{2-2g_v-n_v}b_{g_v,n_v}(\{k_h\mid h\in H(v)\}),&v\in V^{\text{irreg}}
\end{array}\right.
$$
where $y_{1,\alpha_v}=dy(\cp_{\alpha_v})$, $y_{-1,\alpha_v}=(ydx)(\cp_{\alpha_v})$ and $a_{g,n}(k_1,...,k_n)$ and $b_{g,n}(k_1,...,k_n)$---defined in \eqref{eq:abcoeff}---are symmetric functions of $k_i$ so it makes sense to take in a set of cardinality $n$.

{\em Edge weights.}
$$ W(e)=\hbar(2k_{e^+}-1)!!(2k_{e^-}-1)!!B^{\alpha_{e^+},\alpha_{e^-}}_{2k_{e^+},2k_{e^-}}%=\check{B}^{\alpha_e,\alpha'_e}_{k_e,k'_e}
$$
where %$\check{B}^{\alpha,\alpha'}_{k,k'}=(2k-1)!!(2k'-1)!!B^{\alpha,\alpha'}_{2k,2k'}$ and
with respect to the local coordinates $z,z'$ defined by $x=\frac12 z^2+x(\cp_{\alpha})$ and $x=\frac12 z'^2+x(\cp_{\alpha'})$
$$ B^{\alpha,\alpha'}(z,z')=\delta_{\alpha,\alpha'}\frac{dzdz'}{(z-z')^2}+\sum B^{\alpha,\alpha'}_{m,m'}z^mz'^{m'}dzdz'.
$$
Here $\{e^+,e^-\}$ are different orientations of the edge $e$, and by the symmetry $B^{\alpha,\alpha'}_{m,m'}=B^{\alpha',\alpha}_{m',m}$ the weight $B^{\alpha_{e^+},\alpha_{e^-}}_{2k_{e^+},2k_{e^-}}$ depends only on the (unoriented) edge $e$.

%where with respect to the local coordinates $z_i$ defined by $x=\frac12 z_i^2+x(\cp_{\alpha_i})$$$ B^{\alpha_1,\alpha_2}(z_1,z_2)=\delta_{\alpha_1,\alpha_2}\frac{dz_1dz_2}{(z_1-z_2)^2}+\sum B^{\alpha_1,\alpha_2}_{m_1,m_2}z^{m_1}z'^{m_2}dz_1dz_2. $$

{\em Ordinary leaf weights}
$$W(\ell)= V^{\alpha_\ell}_{k_\ell}(p_\ell)$$
The dependence of $\omega_{g,n}$ on $p_i\in\Sigma$ occurs via $V^{\alpha}_{k}(p)$ which is defined in \eqref{Vdiff}.

{\em Dilaton leaf weights}
$$W(\lambda)= \left\{\begin{array}{ll}\displaystyle\frac{1}{y_{1,\alpha_{\lambda}}}(2k_{\lambda}+1)!!y_{2k_{\lambda}+1,\alpha_{\lambda}},&v_{\lambda}\in V^{\text{reg}}\\
\displaystyle\frac{1}{y_{-1,\alpha_{\lambda}}}(2k_{\lambda}+1)!!y_{2k_{\lambda}+1,\alpha_{\lambda}},&v_{\lambda}\in V^{\text{irreg}}
\end{array}\right.
$$
where $y_{2k+1,\alpha}$ are the odd coefficients of the local expansion $y=\sum y_{k,\alpha}z^k$ with respect to the local coordinate $z$ defined by $x=\frac12 z^2+x(\cp_{\alpha})$.
\end{definition}

The recursive structure \eqref{TRrec} can be encoded in graphs \cite{EOrInv} with edges decorated by $\omega_{0,1}(p)$, $\omega_{0,2}(p,p')$ and $K(p,p')$ obtained by applying \eqref{TRrec} repeatedly which corresponds to the construction of a surface of type $(g,n)$ via recursively attaching $2g-2+n$ pairs of  pants.  This enables one to express $\omega_{g,n}$ as a weighted sum over graphs with $2g-2+n$ vertices.

\begin{theorem} \label{graphexp}
For a spectral curve $(\Sigma,B,x,y)$ and $2g-2+n>0$,
$$\omega_{g,n}(p_1,...,p_n)=\frac{1}{n!}\sum_{\gamma\in\Gamma_{g,n}}\prod_{\ v\in V(\gamma)}W(v)\prod_{e\in E(\gamma)}W(e)\prod_{\ell\in L^*(\gamma)}W(\ell)\prod_{\lambda\in L^\bullet(\gamma)}W(\lambda)
$$
with weights $W$ defined in Definition~\ref{weights}.
\end{theorem}
\begin{proof}
The regular case of this theorem is proven in \cite{DOSSIde} and \cite{EynInv}.  Graphs arise in the expression for $\omega_{g,n}$, as originally described in \cite{EOrInv}, as a means of encoding occurrences of the kernels $K(z_1,z_2)$ and $B(z_1,z_2)$ in formulae for the correlators obtained by iterating the recursion \ref{TRrec}.  Importantly the contributions to graphs by $B(p,p')$ are independent of the contributions to graphs by $y$.  Hence the proof of Theorem~\ref{graphexp} in \cite{DOSSIde} adapts immediately to allow irregular vertices, with changes only to vertex weights, and corresponding dilaton leaf weights, for vertices corresponding to irregular zeros of $dx$.  The irregular vertex weights are determined from Proposition~\ref{th:defbes}.
\end{proof}
\begin{remark}
Although the main theorem in \cite{DOSSIde} requires a finiteness assumption on the spectral curve, Theorem~\ref{graphexp} is proven there in full generality for any regular spectral curve.  The theorem is also proven for any regular spectral curve in \cite{EynInv} Proposition 4.1, where the graphical sum is replaced by a moduli space of $D$-coloured stable Riemann surfaces.
\end{remark}

The graphical expansion in Theorem~\ref{graphexp} of the correlators $\omega_{g,n}$ of $S$ can be restated in terms of differential operators acting on the partition function $Z^S$ leading to a proof of the main result Theorem~\ref{th:main}.  The function $R(z)$ and operator $\hat{R}$ below are obtained from \cite{DOSSIde}.

\begin{definition} \label{GivR}
Given a spectral curve $S=(\Sigma,B,x,y)$ define $R(z)=\sum R_kz^k\in\text{End}(V)[[z]]$ for a dimension $D(=$ number of zeros of $dx$) vector space $V$ by
$$\left[R^{-1}(z)\right]^\alpha_\beta = -\frac{\sqrt{z}}{\sqrt{2\pi}}\int_{\Gamma_\beta} B(\cp_\alpha,p)\cdot e^{\frac{(x(\cp_\beta)-x(p))}{ z}}$$
where $\Gamma_\beta$ is a path of steepest descent of $-x(p)/z$ containing $\beta$,
and define the sequence $r_k\in\text{End}(V)$ by $R(z)=\displaystyle\exp(\sum_{\ell>0} r_\ell z^\ell)$.
\end{definition}
\begin{definition} \label{ops}
Given a spectral curve $S=(\Sigma,B,x,y)$ define
\begin{align*}
\hat{R}&=\exp\left\{\sum_{\ell=1}^{\infty}\sum_{\alpha,\beta}
\left(\sum_{k=0}^{\infty}v^{k,\beta}(r_k)_\beta^\alpha\frac{\partial}{\partial v^{k+\ell,\alpha}}+\frac{\hbar}{2}\sum_{m=0}^{\ell-1}(-1)^{m+1}(r_\ell)^{\alpha}_{\beta}\frac{\partial^2}{\partial v^{m,\alpha}\partial v^{\ell-m-1,\beta}}\right)
\right\}\\
\hat{T}&=\exp\left(-\sum_{\alpha=1}^m\displaystyle\sum_{k>0} (2k-1)!!\frac{y_{2k-1,\alpha}}{y_{-1,\alpha}}\frac{\partial}{\partial v^{k,\alpha}}-\sum_{\alpha=m+1}^D\displaystyle\sum_{k>1} (2k-1)!!\frac{y_{2k-1,\alpha}}{y_{1,\alpha}}\frac{\partial}{\partial v^{k,\alpha}}\right)\\
\hat{\Delta}&\quad \text{acts on the $\alpha$th factor $Z^{\text{BGW}}$ by\ } \hbar\mapsto=y_{-1,\alpha}^{-2}\hbar \ (\text{or\ }y_{1,\alpha}^{-2}\hbar, \text{for\ }Z^{\text{KW}})
\end{align*}
We have assumed that $y$ is irregular at $\cp_\alpha$ for $\alpha=1,...,m$ and regular otherwise.  Here $r_k$ is defined in Definition~\ref{GivR} and $y_{2k-1,\alpha}$ are the odd coefficients of the local expansion $y=\sum y_{j,\alpha}z^j$ with respect to the local coordinate $z$ defined by $x=\frac12 z^2+x(\cp_{\alpha})$.
\end{definition}

The following restatement of Propositions~\ref{th:defair} and \ref{th:defbes} is a special case of Theorem~\ref{th:main} and in fact is used in the proof of Theorem~\ref{th:main}.
\begin{proposition}  \label{th:trans}
The partition function of the spectral curve
$$
S=(\Sigma,B,x,y)=(\bp^1,\frac{dzdz'}{(z-z')^2},\frac{1}{2}z^2,\displaystyle\sum_{k=-1}^{\infty}y_kz^k)
$$
is obtained via translation of the appropriate tau function:
$$Z^S=\left\{\begin{array}{ll}Z^{\text{BGW}}(y_{-1}^{-2}\hbar,\{y_{-1}^{-1}v^{0,1},y_{-1}^{-1}v^{k,1}-(2k-1)!!\frac{y_{2k-1}}{y_{-1}},k>0\}),&y_{-1}\neq 0\\
Z^{\text{KW}}(y_{1}^{-2}\hbar,\{y_{-1}^{-1}v^{0,1},y_{-1}^{-1}v^{1,1},y_{1}^{-1}v^{k,1}-(2k-1)!!\frac{y_{2k-1}}{y_{1}},k>1\}),&y_{-1}=0,\ y_1\neq 0.
\end{array}\right.
$$
or equivalently
\begin{equation}  \label{eq:trans}
Z^S=\hat{T}\hat{\Delta}Z^{\text{BGW}}(\hbar,\{v^{k,1}\})\quad {\text or}\quad\hat{T}\hat{\Delta}Z^{\text{KW}}(\hbar,\{v^{k,1}\})
\end{equation}
for
$$ \hat{T}=\exp\left(-\displaystyle\sum_{k>\epsilon} (2k-1)!!\frac{y_{2k-1}}{y_{\text min}}\partial_{v^{k,1}}\right),\ 
\hat{\Delta}(\hbar,\{v^{k,1}\})=(y_{\text min}^{-2}\hbar,\{y_{\text min}^{-1}v^{k,1}\}),\quad (y_{\text min},\epsilon)=(y_{-1},0) {\text\ \ or\ \ }(y_1,1).
$$
\end{proposition}
\begin{proof}
Propositions~\ref{th:defair} and \ref{th:defbes} can be expressed graphically by adding dilaton leaves to each vertex.  The weight of a dilaton leaf with label $i$ is $(2i-1)!!\frac{y_{2i-1}}{y_1}$ at a  regular vertex and $(2i-1)!!\frac{y_{2i-1}}{y_{-1}}$ at an irregular vertex.  But this is also the graphical realisation of translation given by the exponential of a constant vector field.
\end{proof}
\begin{remark}
The regular case of Proposition~\ref{th:trans} is a consequence of work of Manin and Zograf \cite{MZoInv} and Eynard \cite{EynRec} as follows.  Consider
$$F^{\text{KW}}(\hbar,t_0,t_1,...)=\sum_{g,n}\hbar^{g-1}\frac{1}{n!}\sum_{\vec{k}\in\bn^n}\int_{\overline{\modm}_{g,n}}\prod_{i=1}^n\psi_i^{k_i}t_{k_i}
$$
and define a generating function for higher Weil-Petersson volumes
$$F^{\kappa}(\hbar,\vec{t},\vec{s})=\sum_{g,n}\hbar^{g-1}\frac{1}{n!}\sum_{\vec{k}\in\bn^n}\int_{\overline{\modm}_{g,n}}\prod_{i=1}^n\psi_i^{k_i}t_{k_i}\prod_{j=1}^\infty\kappa_j^{m_j}\frac{s_j^{m_j}}{m_j!}.
$$
Manin and Zograf \cite{MZoInv} proved that $F^{\kappa}$ is a translation of $F^{\text{KW}}$
$$F^{\kappa}(\hbar,\vec{t},\vec{s})=F^{\text{KW}}(\hbar,t_0,t_1,t_2+p_1(\vec{s}),...,t_k+p_{k-1}(\vec{s}),...)
$$
where the $p_j$ are (weighted) homogeneous polynomials of degree $j$ defined by
$$ 1-\exp\left(-\sum_{i=1}^\infty s_iz^i\right)=\sum_{j=1}^\infty p_j(s_1,...,s_j)z^j.$$
Eynard \cite{EynRec} proved the same relation between higher Weil-Petersson volumes and topological recursion.  Associate to $y(z)=\displaystyle\sum_{k=1}^{\infty}y_kz^k$ (its Laplace transform)
$$\cl(f)(z)=\frac{1}{y_1}\sum_{k=1}^\infty (2k+1)!!y_{2k+1}z^k=-\sum_{j=1}^\infty p_j(s_1,...,s_j)z^j$$
which defines $s_i$ as a function of $y_{2k+1}$.
Then for $S=(\frac{1}{2}z^2,y(z),\frac{dzdz'}{(z-z')^2})$ he proves $F^S=F^{\kappa}(\hbar,\vec{t},\vec{s})$.  In other words,
$$F^S=F^{\text{KW}}(\hbar,t_0,t_1,t_2-3!!\frac{y_3}{y_1},...,t_{k}-(2k-1)!!\frac{y_{2k-1}}{y_1},...)
$$
which is the regular case of Proposition~\ref{th:trans}.
% There is an extra power of 2 in Eynard's paper - his series is slightly different to Manin-Zograf.  It is $2^{-d_0}$ which can be absorbed into $s_j2^{-j}$ and $p_j2^{-j}$.  This removes the factor of $2^a$ in his $f(z)$ leaving $(2a+1)!!$.
\end{remark}

\begin{proof}[Proof of Theorem~\ref{th:main}]
The graphical expansion in Theorem~\ref{graphexp} corresponds to a Feynman expansion of the action of the exponential of a differential operator with quadratic and linear terms on a product of $D$ functions of single variables.  This viewpoint, i.e. the determination of the differential operators corresponding to the weighted graphical expressions is dealt with thoroughly in \cite{DSSGiv}.  The weights give coefficients of differential operators.  The factor $\prod_{v\in V(\gamma)}W(v)$ corresponds to the product $Z^{\text{BGW}}(\hbar,\{v^{k,1}\})\cdots  Z^{\text{KW}}(\hbar,\{v^{k,D}\})$ of tau functions, the factor $\prod_{e\in E(\gamma)}W(e)$ corresponds to quadratric terms in the differential operator,  the factor $\prod_{\ell\in L^*(\gamma)}W(\ell)$ corresponds to terms $t_j\frac{\partial}{\partial t_k}$ in the differential operator, and the factor $\prod_{\lambda\in L^\bullet(\gamma)}W(\lambda)$ corresponds to translations.

It is crucial here that the definition of $\hat{R}$ depends only on $(\Sigma,B,x)$ which is independent of $y$ and hence also of the property of the spectral curve being regular or irregular.  This allows us to partly follow the proof of the regular case in \cite{DOSSIde} since the same operator $\hat{R}$ is being used.  In \cite{DSSGiv} the equality
\begin{align*}
\exp\left\{\sum_{\ell=1}^{\infty}\sum_{\alpha,\beta}
\left(\sum_{k=0}^{\infty}v^{k,\beta}(r_k)_\beta^\alpha\frac{\partial}{\partial v^{k+\ell,\alpha}}+\frac{\hbar}{2}\sum_{m=0}^{\ell-1}(-1)^{m+1}(r_\ell)^{\alpha}_{\beta}\frac{\partial^2}{\partial v^{m,\alpha}\partial v^{\ell-m-1,\beta}}\right)
\right\}\qquad\qquad\qquad\qquad\\
=\exp\left\{\sum_{\ell=1}^{\infty}\sum_{\alpha,\beta}
\left(\sum_{k=0}^{\infty}v^{k,\beta}(r_k)_\beta^\alpha\frac{\partial}{\partial v^{k+\ell,\alpha}}\right)\right\}
\exp\left\{\frac{\hbar}{2}\sum_{k,\ell\geq 0}(2k-1)!!(2\ell-1)!!B^{\alpha,\beta}_{2k,2\ell}\frac{\partial^2}{\partial v^{k,\alpha}\partial v^{\ell,\beta}}\right\}
\end{align*}
is proven via %a variant of 
the Campbell--Baker--Hausdorff formula.  The general form of the Campbell--Baker--Hausdorff identity that we need pertains to three differential operators with constant coefficient matrices $A$, $A'$, and $B$ (the matrices $A$ and $A'$ are symmetric): ${\mathcal A}=\bigl(\vec{\partial}A\vec{\partial}^{\text{T}}\bigr)$, ${\mathcal A'}=\bigl(\vec{\partial}A'\vec{\partial}^{\text{T}}\bigr)$, and ${\mathcal B}=\bigl(\vec{t}B\vec{\partial}^{\text{T}}\bigr)$ with derivatives acting to the right.  In particular, $[{\mathcal A},[{\mathcal A},{\mathcal B}]]=0$.  Then 
$$
e^{\mathcal B+\mathcal A}=e^{\mathcal B}e^{\mathcal A'}\ \ \hbox{provided}\ \ B^{\text{T}}A'+A'B=A-e^{-B^{\text{T}}}A\,e^{-B},
$$
or equivalently
$$
A'=\int_{0}^1dx\,e^{-xB^{\text{T}}}A\,e^{-xB}.%\ \ \hbox{or}\ \ A=-\sum_{n=1}^\infty e^{nB^{\text{T}}}(B^{\text{T}}A'+A'B) \,e^{nB}.
$$
Via the action of (constant coefficient) quadratic differential operators expressed graphically (also described well in \cite{DSSGiv}) this gives precisely the correct edge weights yielding the operator $\hat{R}$ in \eqref{eq:main}.

The translation term $\hat{T}$ in \eqref{eq:main} follows immediately from Propositions~\ref{th:defair}, \ref{th:defbes} and \ref{th:trans}.  This is a modification to the proof in \cite{DOSSIde} in two ways.  Firstly, the proof in \cite{DOSSIde} uses an operator $\hat{R}_0=\hat{R}\hat{T}_0$ where $\hat{T}_0$ is the translation arising out of a very special choice of $y$---see below for a discussion of this point.  Secondly, here we needed Proposition~\ref{th:defbes} which resulted in terms given by translations of the tau function $Z^{\text{BGW}}$ in addition to translations of the tau function $Z^{\text{KW}}$.
\end{proof}
Note that related decompositions for matrix models appears in \cite{AMMBGW}.
\begin{example}   \label{exspec}
The four decompositions of partition functions \ref{examples} follow from Theorem~\ref{th:main} applied to the following spectral curves.  Each of the spectral curves is rational with common $x$ and $B$ given by the Cauchy kernel:
$$ x=z+\frac{1}{z},\quad  B=\frac{dzdz'}{(z-z')^2}.$$
In particular, the operator $\hat{R}$ is the same in all four examples.  The examples differ by their choice of $y$:
\begin{itemize}
\item $Z^{\text{GW}}$ uses $y=\ln{z}$ --- \cite{DOSSIde,NScGro}.
\item  $Z^{\text{GUE}}$ uses $y=z$ --- \cite{EynTop}.
\item $Z^{\text{Leg}}$ uses $y=\frac{z}{z^2-1}$ --- Section~\ref{sec:legendre}.
\item $Z^{\text{Des}}$ uses $y=\frac{z}{z+1}$ --- \cite{DNoTop}.
\end{itemize}
\end{example}

Theorem~\ref{th:main} generalises a result for regular spectral curves in \cite{DOSSIde} and its proof is modeled on that result.  We end this section with a comparison of the two decomposition formulae applied to regular spectral curves.

\begin{definition}  \label{yfinite}
Given a regular spectral curve $S=(\Sigma,B,x,y)$ define $Y_{\alpha}$ locally in a neighbourhood of a zero $\cp_\alpha$ of $dx$ by the property that it satisfies
$$\frac{\sqrt\zeta}{\sqrt {2\pi}}\int_{\Gamma_\alpha} {dY_{\alpha}}(p)
e^{(x(p)-x(\cp_\alpha))\zeta}
= \sum_{\beta=1}^D dy(\cp_{\beta}) \cdot \frac{-1}{\sqrt {2\pi\zeta}} \int_{\Gamma_\alpha} B(p,\cp_\alpha)e^{(x(p)-x(\cp_\alpha))\zeta}.
$$
The function $Y_{\alpha}$ is well-defined only up to addition of a function of $x$, because the Laplace transform annihilates functions of $x$, but this ambiguity disappears in the odd coefficients of the local expansion of $Y_{\alpha}$ which is all that is needed in the sequel.
\end{definition}
The construction of \cite{DOSSIde} begins with Givental's decomposition of the partition function of (the correlators of) a cohomological field theory (with flat identity) and produces a spectral curve satisfying the restriction that $y=Y_{\alpha}$ for each $\alpha=1,...,D$ (where $y=Y_{\alpha}$ up to local functions of $x$).
\begin{theorem}[\cite{DOSSIde}]
Given a regular spectral curve $S=(\Sigma,B,x,y)$ satisfying $y=Y_{\alpha}$ for each $\alpha=1,...,D$ the partition function $Z^S$ satisfies the decomposition \eqref{eq:main}.
\end{theorem}
The statement of this theorem is the reverse of the result in \cite{DOSSIde}, but it is easily seen to be reversible on spectral curves satisfying the condition on $y$, \cite{DNOPSDub}.  The translation due to local expansions of $y$ is implicit in the statement in \cite{DOSSIde} and appears inside an operator
$$\hat{R}_0=\exp\left\{-(r_m)^\alpha_{\un}\frac{\partial}{\partial v^{m+1,\alpha}}+\sum_{k=0}^{\infty}v^{k,\alpha}(r_m)^\beta_\alpha\frac{\partial}{\partial v^{m+k,\beta}}+\frac{\hbar}{2}\sum_{i=0}^{m-1}(-1)^{i+1}(r_m)^{\alpha,\beta}\frac{\partial^2}{\partial v^{i,\alpha}\partial v^{m-i-1,\beta}}\right\}
$$
(which they denote by $\hat{R}$) where the vector $\un=\{dy(\cp(\alpha)\}$.  Raising of the indices of $r_m$ is required for endomorphisms with respect to a general basis, but in this paper we express $R$ in a basis (known as normalised canonical coordinates \cite{GivSem}) with respect to which the metric is the identity, so upper and lower indices are the same.  It is related to the operator here by $\hat{R}_0=\hat{R}\hat{T}$ for a translation $\hat{T}$.

This viewpoint allows us to interpret the decomposition \eqref{eq:main} in the regular case as Givental's decomposition of the partition function of a cohomological field theory, now without the restriction of flat identity.
\begin{corollary}
For any regular spectral curve $S=(\Sigma,B,x,y)$, $Z^S$ is a partition function for a cohomological field theory.
\end{corollary}

\section{Legendre Ensemble}   \label{sec:legendre}

One application of Theorem~\ref{th:main} applied to the curve $(x^2-4)y^2=1$ is a Givental type decomposition for the partition function of the Legendre ensemble.  
\begin{proposition}  \label{th:leg}
Consider the Legendre ensemble with trivial potential
\begin{equation}  \label{legendre}
\int_{H_N[-2,2]}DM=\frac{1}{N!}\int_{-2}^2\int_{-2}^2...\int_{-2}^2\prod_{i<j}(\lambda_i-\lambda_j)^2d\lambda_1...d\lambda_N.
\end{equation}
The resolvents $W_{g,n}(x_1,...,x_n)$, defined via the large $N$ asymptotic expansion of the cumulant by
$$\left\langle {\rm tr} \left(\frac{1}{x_1-M}\right)\cdots {\rm tr} \left(\frac{1}{x_n-M}\right)\right\rangle^c\stackrel{N\to\infty}{\sim}\sum_{g\geq 0} N^{2-2g-n}W_{g,n}(x_1,...,x_n),$$
satisfy topological recursion for the spectral curve
\begin{equation}  \label{specur}
S=(\bp^1,\ B=\frac{dzdz'}{(z-z')^2},\ x=z+\frac{1}{z},\ y=\frac{z}{z^2-1}).
\end{equation}
In other words $W_{g,n}(x_1,...,x_n)dx_1...dx_n$ gives an expansion of  $\omega_{g,n}$ at $x_i=\infty$.
\end{proposition}
Note that the existence of the asymptotic expansion $W_{n} = \sum_{g \geq 0} N^{2 - 2g - n} W_{g,n}$ is not trivial and cannot be obtained from the loop equations only, but it is a consequence of the results of Borot and Guionnet \cite{BGuAsy}.  We omit the proof of Proposition~\ref{th:leg} which is rather standard, via the usual loop equations.

A corollary of Theorem~\ref{th:main} is:
$$Z^{\text{Leg}}(\{v^{k,1},v^{k,2}\})=\hat{R}\hat{T}\hat{\Delta}Z^{\text{BGW}}(\{v^{k,1}\})\times Z^{\text{BGW}}(\{v^{k,2}\}).$$
%Note that the partition function $Z^{\text{Leg}}$ of the spectral curve \eqref{specur} is equivalent after a change of coordinates to the partition function of the matrix integral \eqref{legendre} with $V(M)= \sum_{r=1}^\infty \tr M^r x_j^{-r-1}$ since $\langle  \tr(M^r) \rangle = Z^{-1}\frac {\partial}{\partial t_r}Z.$

We now consider this example closely and demonstrate how the decomposition \eqref{th:main} produces calculations via pictures realising the graphical treatment of Theorem~\ref{graphexp}.  We need only consider connected graphs to produce terms $F=\log Z$.  In the calculations via pictures we instead use a related decomposition with only edge and vertex weights leading to simpler pictures. It is  analogous to the decomposition proven by the first author in \cite{Ch95} and described next.

\subsection{Chekhov-Givental decomposition}

In \cite{Ch95} the first author obtained a decomposition for the Gaussian model that differs from \eqref{eq:main}.  As in \eqref{eq:main}, it consists of the exponential of a quadratic differential operator acting on two copies of $Z^{\text{KW}}$, with the added translation term present in \eqref{eq:main} but without a rotational term, i.e. it contains only terms of type ${\mathcal A}=\bigl(\vec{\partial}A\vec{\partial}^{\text{T}}\bigr)$ and no terms of type ${\mathcal B}=\bigl(\vec{t}B\vec{\partial}^{\text{T}}\bigr)$, whose absence was ensured by choosing global coordinates properly matching the local ones---see below.  It is equivalent to \eqref{eq:main} under a linear change of coordinates between $v^{k,\alpha}$ and  $\tau^{\pm}_k$ described in \eqref{chcoor}.  For the Legendre model which exhibits hard-edges, we present here an analogous decomposition almost identical to the Gaussian model decomposition in \cite{Ch95}. As in the proof of Theorem~\ref{th:main}, two ingredients of the decomposition arise from $B(p,q)$ -- the Bergman kernel -- and
from the 1-form $ydx$.

(i) $ydx$. We need $ydx$ in the global and local models.
\begin{itemize}
\item[(a)] Global model: $x=e^\lambda+e^{-\lambda}$, $y=1/(e^\lambda-e^{-\lambda})$, $ydx=d\lambda$, \  $\lambda\in \mathbb C$,\ $-\pi/2\le \text{Im\ }\lambda< 3\pi/2$

\item[(b)] Local models: $(x_1,y_1)$ and $(x_2,y_2)$ correspond to two branch points $b_1:\lambda=0$ and $b_2:\lambda=i\pi$.\\
$y_1^2x_1=1/4$, $x_1=\lambda^2$, $y_1dx_1=d\lambda$ for $\lambda\in D_1(0)$ a disk of radius 1 around $0\in\bc$;\\ $y_2^2x_2=1/4$, $x_2=(\lambda-i\pi)^2$, $y_2dx_2=d\lambda$ for $\lambda\in D_1(i\pi)$ a disk of radius 1 around $i\pi\in\bc$.% where again $\lambda$ is the coordinate of the global model. 

\end{itemize}
So, the global and local 1-forms coincide, $ydx^{\text{global}}=ydx^{\text{local}}=d\lambda$, which means that the linear differential part is absent in the Givental decomposition in this case.

(ii) $B(p(\lambda),q(\mu))$. We again have global and local $B$'s:
$$
B(p,q)^{\text{global}}=\frac{de^\lambda de^\mu}{(e^\lambda-e^\mu)^2},\qquad B(p,q)^{\text{local}}=\frac{d\lambda d\mu}{(\lambda-\mu)^2},
$$
defined in the corresponding domains $-\pi/2\le\text{Im\ }\{\lambda,\mu\}<3\pi/2$ and $\lambda,\mu\in D_1(0)$ for the first local model and $\lambda,\mu\in D_1(i\pi)$ for the second local model,
and their corresponding primitives % of the global and local $B$'s, %$e=\int^p \int^q B(\cdot,\cdot)$:
$$
E(p(\lambda),q(\mu))^{\text{global}}=\log (e^\lambda-e^\mu),\qquad E(p(\lambda),q(\mu))^{\text{local}}=\log(\lambda-\mu)
$$
with the same domains of definition as for B's.
We then have that, for $\{\lambda,\mu\}\subset D_1(0)\cup D_1(i\pi)$, so that these points are in the vicinities of the corresponding branching points $b_\mu$ and $b_\lambda$, which can be either $b_1$ or $b_2$,
\begin{align*}
&B(p(\lambda),q(\mu))^{\text{global}}=\frac1{2\pi i} \int_{C_{p(\lambda)}}\frac1{2\pi i}\int_{C_{q(\mu)}} B(p(\lambda),\eta)^{\text{local}}E(\eta,\rho)^{\text{global}}B(\rho,q(\mu))^{\text{local}}\\
&=\delta_{b_\lambda,b_\mu} B(p(\lambda),q(\mu))^{\text{local}} +  \frac1{2\pi i}\int_{C_{p(\lambda)}}\frac1{2\pi i}\int_{C_{q(\mu)}} B(p(\lambda),\eta)^{\text{local}}(E(\eta,\mu)^{\text{global}}-\delta_{b_\lambda,b_\mu}E(\eta,\mu)^{\text{local}}) B(q(\mu),\eta)^{\text{local}}.
\end{align*}
In this expression, the integration contours $C_{p(\lambda)}$ and $C_{q(\mu)}$ are inside the corresponding discs $D_1(b_\lambda)$ and $D_1(b_\mu)$ and, if they are in the same disc, they must separate the points $p(\lambda)$ and $q(\mu)$, the first term in the right hand side describes the propagator of the local model, and the second term corresponds to
the quadratic differential operator, which we consider in details below.

%\tcr{We interpret the propagators $B(p,\eta)^{\text{local}}$ and $B(q,\mu)^{\text{local}}$ as the correlation function endpoints;}

The local times are $t^-_{k}=\frac {1}{\lambda^{2k+1}}$ for $\lambda\in D_1(0)$ and  $t^+_{k}=\frac {1}{(\lambda-i\pi)^{2k+1}}$ for $\lambda\in D_1(i\pi)$;
the corresponding global times are
\begin{equation} \label{t-global}
\tau^\mp_{k}=\frac 1 {(2k)!} \frac{\partial^{2k}}{\partial\lambda^{2k}} \frac{1}{\pm e^\lambda-1}.
\end{equation}
The times $\tau^\pm_{k}$ are related by a linear change of coordinates to the times $v^{\alpha}_k$ used in Theorem~\ref{th:main} as follows.   The global meromorphic differentials $d\tau^\pm_{k}(\lambda)$, defined on $\overline{\bc}$, have poles only at $e^{\lambda}=\pm1$ and are skew-invariant under the involution $\lambda\mapsto-\lambda$.  Hence they lie in the vector space spanned by the differentials $V^\alpha_j(p(\lambda))$ defined in \eqref{Vdiff} and $d\tau^\pm_{k}(\lambda)=\sum_{j,\alpha}  c^{\pm,j}_{\alpha,k}V^\alpha_j(p(\lambda))$.  This defines the linear relation between the times:
\begin{equation} \label{tauv}
\tau^\pm_{k}=\sum_{j,\alpha}  c^{\pm,j}_{\alpha,k}v^\alpha_j.
\end{equation}
The correlation function expansion is $W(\dots, x_j, \dots)\sim \sum_{r=1}^\infty \tr H^r x_j^{-r-1}$ and recalling that
$$
\langle \dots \tr(H^r)\dots \rangle = r\frac {\partial}{\partial t_r}\langle \dots \rangle
$$
we obtain that the expansion of the local correlation
function in the local coordinates has the form
$$
W(\dots, x_j, \dots) = \sum_{r=1}^\infty r x_j^{-r-1}  \frac {\partial}{\partial t_r}.
$$
%Here $x_j$ are local coordinates ($\eta$ and $\mu$ in our case).

When $b_\lambda=b_\mu$,
the coefficient expansion of the quadratic differential operator stems from that of the function
$$
E(\eta,\mu)^{\text{global}}-E(\eta,\mu)^{\text{local}}=\log\frac{e^\eta-e^\mu}{\eta-\mu}
=\frac{\eta+\mu}2+\sum_{m=2}^\infty B_m \frac{(\eta-\mu)^m}{m \cdot m!},
$$
where $B_m$ are the Bernoulli numbers (recall that $B_{2s+1}=0$ for $s\ge 1$),
and with the accounting of the skew-symmetrisation under the transformations $\eta\leftrightarrow -\eta$ and $\mu\leftrightarrow -\mu$,
we obtain the expansion
\begin{equation} \label{B++}
-\sum_{m=0}^\infty \frac{B_{2m+2}}{(2m+2)}\sum_{k=0}^m \frac{\eta^{2k+1}}{(2k+1)!} \frac{\mu^{2(m-k)+1}}{(2(m-k)+1)!}
\end{equation}
with only odd powers of $\mu$ and $\eta$, so the corresponding differential operator, after evaluating the residues in $\eta$ and $\mu$, will contain
derivatives only in odd times, as expected. This operator has the following form  (where we recall that in the final expression we have to substitute local times by the global ones):
\begin{equation} \label{A++}
A_{\pm,\pm}=-\frac12 \sum_{m=0}^\infty \frac{B_{2m+2}}{(2m+2)}\sum_{k=0}^m \frac{(2k+1)\partial/\partial \tau^\pm_{k}}{(2k+1)!}
\frac{(2(m-k)+1) \partial/\partial \tau^\pm_{m-k}}{(2(m-k)+1)!}.
\end{equation}

For different branch points, $E(\eta,\mu)^{\text{global}}=\log(e^\eta-e^{\mu-i\pi})=\log(e^\eta+e^\mu)$ and the corresponding expansion is
$$
\log(e^\eta+e^\mu)=\frac{\eta+\mu}{2}+\sum_{m=2}^\infty \frac{B_m}{m}(1-2^{m})\frac{(\eta-\mu)^m}{m!}
$$
so after the skew-symmetrisation and substitution of differential operators in times, we obtain
\begin{equation} \label{A+-}
A_{+,-}=- \sum_{m=0}^\infty \frac{B_{2m+2}}{(2m+2)}(1-2^{2m+2})\sum_{k=0}^m \frac{(2k+1)\partial/\partial \tau^+_{k}}{(2k+1)!}
\frac{(2(m-k)+1) \partial/\partial \tau^-_{m-k}}{(2(m-k)+1)!}.
\end{equation}
The differential operators $A_{\pm,\pm}$ and $A_{+,-}$ exactly coincide with the corresponding operators from \cite{Ch95} where the analogous transformation was first derived in application to the Gaussian model.

%We therefore have the following theorem which proves
%$$ Z^{HE}(t_i,s_i)={\mathcal D}\left\{Z^{BGW}(t_i+s_i)\times Z^{BGW}(t_i-s_i)\right\}.$$

\begin{theorem}   \label{Fleg}
Consider the free energy of the Legendre ensemble on the interval $[-2,2]$
with the potential $V(H)=\sum_{k=1}^\infty \frac {{\mathfrak t}_k}{k} H^k$ where the times ${\mathfrak t}_k$ are expressed using the Miwa-type transform
${\mathfrak t}_k=\sum_i (e^{\lambda_i}+e^{-\lambda_i})^{-k}$ as $\lambda\to +\infty$. The free energy of this model has the $1/N$ expansion of the form
$F^{\text{Leg}}=\sum_{g=0}^\infty N^{2-2g} F^{\text{Leg}}_g$ where the term $F^{\text{Leg}}_0$ is a quadratic polynomial in ${\mathfrak t}_k$, and
this term is to be interpreted as the normalisation. We then have the exact relation
\begin{equation} \label{ChG}
e^{F^{\text{Leg}}-N^2 F_0^{\text{Leg}}}=e^{A_{+,+}+A_{-,-}+A_{+,-}} Z^{\text{BGW}}(\tau^+(\lambda)) Z^{\text{BGW}}(\tau^-(\lambda)),
\end{equation}
where $Z^{\text{BGW}}(\tau^\pm(\lambda))$ are partition functions of the (local) BGW model (exponentials of tau-functions of the KdV hierarchy) expressed
in terms of times (\ref{t-global}) and $A_{\cdot,\cdot}$ are quadratic differential operators (\ref{A++}) and (\ref{A+-}).
\end{theorem}

\begin{remark}
Expressing the free energy $F^{\text{Leg}}$ in terms of times ${\mathfrak t}_k$ always results in infinite-term expansions for any model even if we restrict ourselves to finite-order monomials corresponding to finite-order correlation functions $W_g(x_1,\dots,x_n$ in the corresponding free theory. It is however well-known (see, e.g., ACNP1) that these expressions are comprised of a finite numbers of terms if we express them in variables $t_k^0:=(e^\lambda-e^{-\lambda})^{-2k-1}$ and $t_k^1:=(e^\lambda+e^{-\lambda})(e^\lambda-e^{-\lambda})^{-2k-1}$, $k=0,1,\dots$.  It is easy to see that all $t_k^{\alpha}$, $\alpha=0,1$, are finite linear combinations of $\tau^\pm_k$.
\end{remark}
We demonstrate Theorem~\ref{Fleg} by calculating low genus examples.  Unlike the Gaussian case, we see that the free energy $F_g^{\text{Leg}}$ is in some sense finite.  It is a finite sum of rational functions apparent in the exact formulae below.  %Furthermore, $F_g^{\text{Leg}}|_{t_0^0=0=t_0^1}$ is a polynomial, in particular it has finitely many terms.  Moreover, $F_g^{\text{Leg}}|_{t_0^0=0=t_0^1}$ uniquely determines $F_g^{\text{Leg}}$ via the so-called dilaton and divisor equations.
\begin{example}\label{F1}
In genus 1, the only contributing graphs have no edges.
$$
\begin{pspicture}(-8,-1)(8,1)
{\psset{unit=0.7}
\newcommand{\PATTERN}[1]{%
{\psset{unit=1}
\pscircle[fillstyle=solid,fillcolor=yellow](0,0){1}
\pscircle[fillstyle=solid,fillcolor=white,linecolor=white](0,0){0.4}
\rput(0,0){\makebox(0,0)[cc]{\small #1}}
}
}
\rput(-10,0){\makebox(0,0)[rc]{$F_1^{\text{Leg}}=\sum\limits_{\pm}$}}
\rput(-8.5,0){\PATTERN{$1_\pm$}}

\rput(-2,0){$=\frac{1}{8}(1-\log(1-\tau_0^+))+\frac{1}{8}(1-\log(1-\tau_0^-))$}
}
\end{pspicture}
$$%
\end{example}

\ifdefined\pictures
\begin{example}\label{F2}
We now consider the decomposition for the $F_2$ term of the free energy.

$$
\begin{pspicture}(-8,-1)(8,1)
{\psset{unit=0.7}
\newcommand{\PATTERN}[1]{%
{\psset{unit=1}
\pscircle[fillstyle=solid,fillcolor=yellow](0,0){1}
\pscircle[fillstyle=solid,fillcolor=white,linecolor=white](0,0){0.4}
\rput(0,0){\makebox(0,0)[cc]{\small #1}}
}

}
\rput(-10,0){\makebox(0,0)[rc]{$F_2^{\text{Leg}}=\sum\limits_{\pm}$}}
\rput(-8.5,0){\PATTERN{$2_\pm$}}
\rput(-7,0){\makebox(0,0)[lc]{$+\sum\limits_{\pm}$}}
\rput(-3,0){
\rput(-1.7,0){\PATTERN{$1_\pm$}}
\rput(1.7,0){\PATTERN{$1_\pm$}}
\psline[linewidth=2pt,linecolor=blue,linestyle=dashed](-1,0)(1,0)
\rput(0,.2){\makebox(0,0)[cb]{\small $A_{\pm,\pm}^{0,0}$}}
}
\rput(0,0){\makebox(0,0)[cc]{$+$}}
\rput(3,0){
\rput(-1.7,0){\PATTERN{$1_+$}}
\rput(1.7,0){\PATTERN{$1_-$}}
\psline[linewidth=2pt,linecolor=blue,linestyle=dashed](-1,0)(1,0)
\rput(0,.2){\makebox(0,0)[cb]{\small $A_{+,-}^{0,0}$}}
}
\rput(5.7,0){\makebox(0,0)[lc]{$+\sum\limits_{\pm}$}}
\rput(8,0){\PATTERN{$1_\pm$}}
\rput(9,0){\psarc[linewidth=2pt,linecolor=blue,linestyle=dashed](0,0){0.9}{-140}{140}}
\rput(10,0){\makebox(0,0)[lc]{\small $A_{\pm,\pm}^{0,0}$}}
}
\end{pspicture}
$$
Here
$$
\begin{pspicture}(-8,-1)(8,1)
{\psset{unit=0.7}
\newcommand{\PATTERN}[1]{%
{\psset{unit=1}
\pscircle[fillstyle=solid,fillcolor=yellow](0,0){1}
\pscircle[fillstyle=solid,fillcolor=white,linecolor=white](0,0){0.4}
\rput(0,0){\makebox(0,0)[cc]{\small #1}}
}
}
\rput(-10.5,0){\PATTERN{$2_\pm$}}
\rput(-9.3,0){\makebox(0,0)[lc]{\small $=\dfrac {3\tau_1^\pm}{256(1-\tau_0^\pm)^3}$,}}
\rput(-2,0){
\rput(-1.7,0){\PATTERN{$1_\pm$}}
\psline[linewidth=2pt,linecolor=blue,linestyle=dashed](-1,0)(-0.2,0)
\rput(0.2,0){\makebox(0,0)[lc]{\small $=\dfrac{\partial F_1^{\text{Leg}}}{\partial \tau_0^\pm}=\dfrac {1}{8(1-\tau_0^\pm)}$,}}
}
\rput(6,0){
\rput(-1.7,0){\PATTERN{$1_\pm$}}
\psline[linewidth=2pt,linecolor=blue,linestyle=dashed](-1.2,0.5)(-0.7,1)
\psline[linewidth=2pt,linecolor=blue,linestyle=dashed](-1.2,-0.5)(-0.7,-1)
\rput(-0.2,0){\makebox(0,0)[lc]{\small $=\dfrac{\partial^2 F_1^{\text{Leg}}}{\partial (\tau_0^\pm)^2}=\dfrac {1}{8(1-\tau_0^\pm)^2}$,}}
}
}
\end{pspicture}
$$

$$
A_{\pm,\pm}^{k,l}=-\frac 12\cdot\frac{B_{2(k+l+1)}}{2(k+l+1)(2k)!(2l)!},\qquad
A_{+,-}^{k,l}=- \frac{B_{2(k+l+1)}}{2(k+l+1)(2k)!(2l)!}(2^{2(k+l+1)}-1).
$$
Recalling that $B_2=1/6$, we have that $A_{\pm,\pm}^{0,0}=-1/24$ and $A_{+,-}^{0,0}=-1/4$, so the combination of the second and fourth terms in the graphical expansion for $F_2^{\text{Leg}}$ yield
$$
\Bigl(-\frac{1}{8^2\cdot 24}-\frac 1{8\cdot 24}\Bigr)\frac 1{(1-\tau_0^\pm)^2}= -\frac{3}{512} \frac 1{(1-\tau_0^\pm)^2}
$$
and the third term  of the same graphical expansion yields
$$
-\frac{1}{8^2\cdot 4}\frac{1}{(1-\tau_0^+)(1-\tau_0^-)}=-\frac{1}{256} \frac{1}{(1-\tau_0^+)(1-\tau_0^-)}.
$$
Hence
$$
F_2^{\text{Leg}}=\frac {3\tau_1^\pm}{256(1-\tau_0^\pm)^3} -\frac{3}{512} \frac 1{(1-\tau_0^\pm)^2}-\frac{1}{256} \frac{1}{(1-\tau_0^+)(1-\tau_0^-)}.
$$

% The whole first line, as was demonstrated above, just combines into $\dfrac{3\tau_1^\pm}{256(1-\tau_0^\pm)^3}$ (recall that normalisations for $t_k$ and $\tau_k$ differ by the factor of $4^k$).}\tcb{\bf REMOVE?}
\end{example}

\begin{example}\label{F3}
We now consider the decomposition for $F_3$ term of the free energy.
$$
\begin{pspicture}(-8,-1)(8,1)
{\psset{unit=0.7}
\newcommand{\PATTERN}[1]{%
{\psset{unit=1}
\pscircle[fillstyle=solid,fillcolor=yellow](0,0){1}
\pscircle[fillstyle=solid,fillcolor=white,linecolor=white](0,0){0.4}
\rput(0,0){\makebox(0,0)[cc]{\small #1}}
}
}
\rput(-10,0){\makebox(0,0)[rc]{$F_3^{\text{Leg}}=\sum\limits_{\pm}$}}
\rput(-8.5,0){\PATTERN{$3_\pm$}}
\rput(-7.5,0){\makebox(0,0)[lc]{$+\sum\limits_{\pm}{\tcb{2}}$}}
\rput(-3.2,0){
\rput(-1.7,0){\PATTERN{$2_\pm$}}
\rput(1.7,0){\PATTERN{$1_\pm$}}
\psline[linewidth=2pt,linecolor=blue,linestyle=dashed](-1,0)(1,0)
\rput(0,.2){\makebox(0,0)[cb]{\small $A_{\pm,\pm}^{0,0}$}}
}
\rput(0,0){\makebox(0,0)[cc]{$+{\tcb{1}}$}}
\rput(3.2,0){
\rput(-1.7,0){\PATTERN{$2_\pm$}}
\rput(1.7,0){\PATTERN{$1_\mp$}}
\psline[linewidth=2pt,linecolor=blue,linestyle=dashed](-1,0)(1,0)
\rput(0,.2){\makebox(0,0)[cb]{\small $A_{+,-}^{0,0}$}}
}
\rput(6,0){\makebox(0,0)[lc]{$+\sum\limits_{\pm}{\tcb{1}}$}}
\rput(8.5,0){\PATTERN{$2_\pm$}}
\rput(9.5,0){\psarc[linewidth=2pt,linecolor=blue,linestyle=dashed](0,0){0.9}{-140}{140}}
\rput(10.5,0){\makebox(0,0)[lc]{\small $A_{\pm,\pm}^{0,0}$}}
}
\end{pspicture}
$$
$$
\begin{pspicture}(-8,-1)(8,1)
{\psset{unit=0.7}
\newcommand{\PATTERN}[1]{%
{\psset{unit=1}
\pscircle[fillstyle=solid,fillcolor=yellow](0,0){1}
\pscircle[fillstyle=solid,fillcolor=white,linecolor=white](0,0){0.4}
\rput(0,0){\makebox(0,0)[cc]{\small #1}}
}
}
%\rput(-10,0){\makebox(0,0)[rc]{$+\sum\limits_{\pm}$}}
%\rput(-8.5,0){\PATTERN{$3_\pm$}}
\rput(-10,0){\makebox(0,0)[lc]{$+\sum\limits_{\pm}{\tcb{2}\times\tcr{2}}$}}
\rput(-5,0){
\rput(-1.7,0){\PATTERN{$2_\pm$}}
\rput(1.7,0){\PATTERN{$1_\pm$}}
\rput(-1.1,0){\makebox(0,0)[rc]{\small $1$}}
\psline[linewidth=2pt,linecolor=red,linestyle=dashed](-1,0)(1,0)
\rput(0,.2){\makebox(0,0)[cb]{\small $A_{\pm,\pm}^{1,0}$}}
}
\rput(-1.5,0){\makebox(0,0)[cc]{$+{\tcb{1}\times\tcr{2}}$}}
\rput(2.3,0){
\rput(-1.7,0){\PATTERN{$2_\pm$}}
\rput(1.7,0){\PATTERN{$1_\mp$}}
\rput(-1.1,0){\makebox(0,0)[rc]{\small $1$}}
\psline[linewidth=2pt,linecolor=red,linestyle=dashed](-1,0)(1,0)
\rput(0,.2){\makebox(0,0)[cb]{\small $A_{\pm,\mp}^{1,0}$}}
}
\rput(5,0){\makebox(0,0)[lc]{$+\sum\limits_{\pm}{\tcb{2}\times\tcr{2}}$}}
\rput(8.3,0){\PATTERN{$2_\pm$}}
\rput(8.5,0.6){\makebox(0,0)[rc]{\small $1$}}
\rput(9.3,0){\psarc[linewidth=2pt,linecolor=red,linestyle=dashed](0,0){0.9}{-140}{140}}
\rput(10.3,0){\makebox(0,0)[lc]{\small $A_{\pm,\pm}^{1,0}$}}
}
\end{pspicture}
$$
$$
\begin{pspicture}(-8,-1)(8,1)
{\psset{unit=0.7}
\newcommand{\PATTERN}[1]{%
{\psset{unit=1}
\pscircle[fillstyle=solid,fillcolor=yellow](0,0){1}
\pscircle[fillstyle=solid,fillcolor=white,linecolor=white](0,0){0.4}
\rput(0,0){\makebox(0,0)[cc]{\small #1}}
}
}
\rput(-11.5,0){\makebox(0,0)[rc]{$+\sum\limits_{\pm}{\tcb{1/2}}$}}
\rput(-9.5,0){\PATTERN{$1_\pm$}}
\rput(-8.5,0){\psarc[linewidth=2pt,linecolor=blue,linestyle=dashed](0,0){0.9}{-140}{140}}
\rput(-10.5,0){\psarc[linewidth=2pt,linecolor=blue,linestyle=dashed](0,0){0.9}{40}{320}}
\rput(-10.5,1){\makebox(0,0)[rb]{\small $A_{\pm,\pm}^{0,0}$}}
\rput(-8.5,1){\makebox(0,0)[lb]{\small $A_{\pm,\pm}^{0,0}$}}
\rput(-7,0){\makebox(0,0)[lc]{$+\sum\limits_{\pm}{\tcb{2}}$}}
\rput(-2,0){
\rput(-1.7,0){\PATTERN{$1_\pm$}}
\rput(-2.7,0){\psarc[linewidth=2pt,linecolor=blue,linestyle=dashed](0,0){0.9}{40}{320}}
\rput(-2.7,1){\makebox(0,0)[rb]{\small $A_{\pm,\pm}^{0,0}$}}
\rput(1.7,0){\PATTERN{$1_\pm$}}
%\rput(-1.1,0){\makebox(0,0)[rc]{\small $1$}}
\psline[linewidth=2pt,linecolor=blue,linestyle=dashed](-1,0)(1,0)
\rput(0,.2){\makebox(0,0)[cb]{\small $A_{\pm,\pm}^{0,0}$}}
}
\rput(6,0){
\rput(-1.7,0){\PATTERN{$1_\pm$}}
\rput(-2.7,0){\psarc[linewidth=2pt,linecolor=blue,linestyle=dashed](0,0){0.9}{40}{320}}
\rput(-2.7,1){\makebox(0,0)[rb]{\small $A_{\pm,\pm}^{0,0}$}}
\rput(1.7,0){\PATTERN{$1_\mp$}}
%\rput(-1.1,0){\makebox(0,0)[rc]{\small $1$}}
\psline[linewidth=2pt,linecolor=blue,linestyle=dashed](-1,0)(1,0)
\rput(0,.2){\makebox(0,0)[cb]{\small $A_{\pm,\mp}^{0,0}$}}
}
\rput(1.4,0){\makebox(0,0)[cc]{$+\sum\limits_{\pm}{\tcb{1}}$}}
%
%\rput(5.7,0){\makebox(0,0)[lc]{$+\sum\limits_{\pm}$}}
%
%\rput(8,0){\PATTERN{$2_\pm$}}
%\rput(8.2,0.6){\makebox(0,0)[rc]{\small $1$}}
%\rput(9,0){\psarc[linewidth=2pt,linecolor=red,linestyle=dashed](0,0){0.9}{-140}{140}}
%\rput(10,0){\makebox(0,0)[lc]{\small $A_{\pm,\pm}^{1,0}$}}
%
%
}
\end{pspicture}
$$
$$
\begin{pspicture}(-8,-1)(8,1)
{\psset{unit=0.7}
\newcommand{\PATTERN}[1]{%
{\psset{unit=1}
\pscircle[fillstyle=solid,fillcolor=yellow](0,0){1}
\pscircle[fillstyle=solid,fillcolor=white,linecolor=white](0,0){0.4}
\rput(0,0){\makebox(0,0)[cc]{\small #1}}
}
}
\rput(-12,0){\makebox(0,0)[rc]{$+\sum\limits_{\pm}{\tcb{1}}$}}
\rput(-9,0){
\rput(-1.7,0){\PATTERN{$1_\pm$}}
\rput(1.7,0){\PATTERN{$1_\pm$}}
%\rput(-1.1,0){\makebox(0,0)[rc]{\small $1$}}
\psline[linewidth=2pt,linecolor=blue,linestyle=dashed](-1,0.3)(1,0.3)
\psline[linewidth=2pt,linecolor=blue,linestyle=dashed](-1,-0.3)(1,-0.3)
\rput(0,.5){\makebox(0,0)[cb]{\small $A_{\pm,\pm}^{0,0}$}}
\rput(0,-.5){\makebox(0,0)[ct]{\small $A_{\pm,\pm}^{0,0}$}}
}
\rput(-6.3,0){\makebox(0,0)[lc]{$+{\tcb{1/2}}$}}
\rput(-2.2,0){
\rput(-1.7,0){\PATTERN{$1_+$}}
\rput(1.7,0){\PATTERN{$1_-$}}
%\rput(-1.1,0){\makebox(0,0)[rc]{\small $1$}}
\psline[linewidth=2pt,linecolor=blue,linestyle=dashed](-1,0.3)(1,0.3)
\psline[linewidth=2pt,linecolor=blue,linestyle=dashed](-1,-0.3)(1,-0.3)
\rput(0,.5){\makebox(0,0)[cb]{\small $A_{+,-}^{0,0}$}}
\rput(0,-.5){\makebox(0,0)[ct]{\small $A_{+,-}^{0,0}$}}
}
\rput(1,0){\makebox(0,0)[cc]{$+\sum\limits_{\pm}{\tcb{2}}$}}
\rput(6.1,0){
\rput(-3.4,0){\PATTERN{$1_\pm$}}
\rput(0,0){\PATTERN{$1_\pm$}}
\rput(3.4,0){\PATTERN{$1_\pm$}}
%\rput(-1.1,0){\makebox(0,0)[rc]{\small $1$}}
\psline[linewidth=2pt,linecolor=blue,linestyle=dashed](-2.7,0)(-0.7,0)
\psline[linewidth=2pt,linecolor=blue,linestyle=dashed](2.7,0)(0.7,0)
\rput(-1.7,.2){\makebox(0,0)[cb]{\small $A_{\pm,\pm}^{0,0}$}}
\rput(1.7,.2){\makebox(0,0)[cb]{\small $A_{\pm,\pm}^{0,0}$}}
}
%
%\rput(5.7,0){\makebox(0,0)[lc]{$+\sum\limits_{\pm}$}}
%
%\rput(8,0){\PATTERN{$2_\pm$}}
%\rput(8.2,0.6){\makebox(0,0)[rc]{\small $1$}}
%\rput(9,0){\psarc[linewidth=2pt,linecolor=red,linestyle=dashed](0,0){0.9}{-140}{140}}
%\rput(10,0){\makebox(0,0)[lc]{\small $A_{\pm,\pm}^{1,0}$}}
%
%
}
\end{pspicture}
$$
$$
\begin{pspicture}(-8,-1)(8,1)
{\psset{unit=0.7}
\newcommand{\PATTERN}[1]{%
{\psset{unit=1}
\pscircle[fillstyle=solid,fillcolor=yellow](0,0){1}
\pscircle[fillstyle=solid,fillcolor=white,linecolor=white](0,0){0.4}
\rput(0,0){\makebox(0,0)[cc]{\small #1}}
}
}
\rput(-10.5,0){\makebox(0,0)[rc]{$+\sum\limits_{\pm}{\tcb{2}}$}}
\rput(-5.5,0){
\rput(-3.4,0){\PATTERN{$1_\pm$}}
\rput(0,0){\PATTERN{$1_\pm$}}
\rput(3.4,0){\PATTERN{$1_\mp$}}
%\rput(-1.1,0){\makebox(0,0)[rc]{\small $1$}}
\psline[linewidth=2pt,linecolor=blue,linestyle=dashed](-2.7,0)(-0.7,0)
\psline[linewidth=2pt,linecolor=blue,linestyle=dashed](2.7,0)(0.7,0)
\rput(-1.7,.2){\makebox(0,0)[cb]{\small $A_{\pm,\pm}^{0,0}$}}
\rput(1.7,.2){\makebox(0,0)[cb]{\small $A_{\pm,\mp}^{0,0}$}}
}
\rput(-.2,0){\makebox(0,0)[cc]{$+\sum\limits_{\pm}{\tcb{1/2}}$}}
\rput(5.5,0){
\rput(-3.4,0){\PATTERN{$1_\pm$}}
\rput(0,0){\PATTERN{$1_\mp$}}
\rput(3.4,0){\PATTERN{$1_\pm$}}
%\rput(-1.1,0){\makebox(0,0)[rc]{\small $1$}}
\psline[linewidth=2pt,linecolor=blue,linestyle=dashed](-2.7,0)(-0.7,0)
\psline[linewidth=2pt,linecolor=blue,linestyle=dashed](2.7,0)(0.7,0)
\rput(-1.7,.2){\makebox(0,0)[cb]{\small $A_{\pm,\mp}^{0,0}$}}
\rput(1.7,.2){\makebox(0,0)[cb]{\small $A_{\mp,\pm}^{0,0}$}}
}
%
%\rput(5.7,0){\makebox(0,0)[lc]{$+\sum\limits_{\pm}$}}
%
%\rput(8,0){\PATTERN{$2_\pm$}}
%\rput(8.2,0.6){\makebox(0,0)[rc]{\small $1$}}
%\rput(9,0){\psarc[linewidth=2pt,linecolor=red,linestyle=dashed](0,0){0.9}{-140}{140}}
%\rput(10,0){\makebox(0,0)[lc]{\small $A_{\pm,\pm}^{1,0}$}}
%
%
}
\end{pspicture}
$$
Here
$$
\begin{pspicture}(-8,-1)(8,1)
{\psset{unit=0.7}
\newcommand{\PATTERN}[1]{%
{\psset{unit=1}
\pscircle[fillstyle=solid,fillcolor=yellow](0,0){1}
\pscircle[fillstyle=solid,fillcolor=white,linecolor=white](0,0){0.4}
\rput(0,0){\makebox(0,0)[cc]{\small #1}}
}
}
\rput(-10.5,0){\PATTERN{$2_\pm$}}
\psline[linewidth=2pt,linecolor=red,linestyle=dashed](-9.8,0)(-9,0)
\rput(-9.9,0){\makebox(0,0)[rc]{\small $1$}}
\rput(-8.8,0){\makebox(0,0)[lc]{\small $=\dfrac {3}{256(1-\tau_0^\pm)^3}$,}}
\rput(6.5,0){
\rput(-10.5,0){\PATTERN{$2_\pm$}}
\psline[linewidth=2pt,linecolor=red,linestyle=dashed](-9.8,0.3)(-9,0.3)
\rput(-9.9,0.3){\makebox(0,0)[rc]{\small $1$}}
\psline[linewidth=2pt,linecolor=red,linestyle=dashed](-9.8,-0.3)(-9,-0.3)
\rput(-8.8,0){\makebox(0,0)[lc]{\small $=\dfrac {3\cdot 3}{256(1-\tau_0^\pm)^4}$,}}
}
\rput(13,0){
\rput(-10.5,0){\PATTERN{$2_\pm$}}
\psline[linewidth=2pt,linecolor=blue,linestyle=dashed](-9.8,0.3)(-9,0.3)
\psline[linewidth=2pt,linecolor=blue,linestyle=dashed](-9.8,-0.3)(-9,-0.3)
\rput(-8.8,0){\makebox(0,0)[lc]{\small $=\dfrac {3\cdot 3\cdot 4 \tau_1^\pm}{256(1-\tau_0^\pm)^5}$,}}
}
%\rput(-2,0){
%\rput(-1.7,0){\PATTERN{$1_\pm$}}
%\psline[linewidth=2pt,linecolor=blue,linestyle=dashed](-1,0)(-0.2,0)
%\rput(0.2,0){\makebox(0,0)[lc]{\small $=\dfrac{\partial F_1^{\text{Leg}}}{\partial \tau_0^\pm}=\dfrac {1}{8(1-\tau_0^\pm)}$,}}
%}
%
%
}
\end{pspicture}
$$
and
$$
\begin{pspicture}(-8,-1)(8,1)
{\psset{unit=0.7}
\newcommand{\PATTERN}[1]{%
{\psset{unit=1}
\pscircle[fillstyle=solid,fillcolor=yellow](0,0){1}
\pscircle[fillstyle=solid,fillcolor=white,linecolor=white](0,0){0.4}
\rput(0,0){\makebox(0,0)[cc]{\small #1}}
}
}
\rput(0,0){
\rput(-1.7,0){\PATTERN{$1_\pm$}}
\psline[linewidth=2pt,linecolor=blue,linestyle=dashed](-1.2,0.5)(-0.7,1)
\psline[linewidth=2pt,linecolor=blue,linestyle=dashed](-1.2,-0.5)(-0.7,-1)
\rput(-1.7,0){\psarc[linewidth=3pt,linecolor=blue,linestyle=dotted](0,0){1.3}{-40}{40}}
\rput(-0.2,0){\makebox(0,0)[lc]{$k$}}
\rput(0.2,0){\makebox(0,0)[lc]{\small $=\dfrac{\partial^k F_1^{\text{Leg}}}{\partial (\tau_0^\pm)^k}=\dfrac {(k-1)!}{8(1-\tau_0^\pm)^k}$,}}
}
}
\end{pspicture}
$$
$$
A_{\pm,\pm}^{k,l}=-\frac 12\cdot\frac{B_{2(k+l+1)}}{2(k+l+1)(2k)!(2l)!},\qquad
A_{+,-}^{k,l}=- \frac{B_{2(k+l+1)}}{2(k+l+1)(2k)!(2l)!}(2^{2(k+l+1)}-1).
$$
$F_3^{\text{BGW}}=C_1 \tau_2/(1-\tau_0)^6+C_2 (\tau_1)^2/(1-\tau_0)^7$. Here $\tau_2^\pm$ is obtained by a Laplace transform from a ``pure'' state corresponding to the monomial $b^4(\mp 1)^b$.
Recalling that $B_2=1/6$ and $B_4=-1/30$, we have that $A_{\pm,\pm}^{0,0}=-1/24$, $A_{+,-}^{0,0}=-1/4$, $A_{\pm,\pm}^{1,0}=1/(30\cdot 16)$, $A_{\pm,\mp}^{1,0}=1/16$.
so we fix everything indicating in gray the symmetry factors of diagrams. All derivatives in the times $\tau_1^\pm$ are to be multiplied by factors of two in order to obtain $F_3$.

The terms in the first line are
$$
F_3^{\text{BGW}}(\tau_+)+F_3^{\text{BGW}}(\tau_-)-\frac{2\cdot3\cdot 3\cdot \tau_1^\pm}{24\cdot 256\cdot 8 (1-\tau_0^\pm)^5}
-\frac{3\cdot 3\cdot \tau_1^\pm}{4\cdot 256\cdot 8 (1-\tau_0^\pm)^4(1-\tau_0^\mp)} -\frac{3\cdot 3\cdot 4\cdot \tau_1^\pm}{24\cdot 256(1-\tau_0^\pm)^5}
$$
and all other terms depend only on zeroth times and are a linear combination of $(1-\tau_0^+)^{-4}+(1-\tau_0^-)^{-4}$, $(1-\tau_0^+)^{-3}(1-\tau_0^-)^{-1}
+(1-\tau_0^-)^{-3}(1-\tau_0^+)^{-1}$, and $(1-\tau_0^+)^{-2}(1-\tau_0^-)^{-2}$ with necessarily positive coefficients.
\end{example}
\fi

\setcounter{section}{0}
\appendix{Quantisation}

In this section we give a brief background for the construction of $\hat{R}$ from $R$ following Givental \cite{GivGro}.  We first consider quantisation in finite dimensions which easily generalises  to infinite dimensions.

Consider the standard holomorphic form on $\bc^{2N}=T^*\bc^N$ given by $\omega=\sum dp_\alpha\wedge dq_\alpha$ with Darboux coordinates $\{q_\alpha,p_\alpha\}$.  A transformation $A:\bc^{2N}\to\bc^{2N}$ that preserves the symplectic form is {\em symplectic}.  We will consider only linear symplectic transformations which correspond to matrices $A\in\text{Sp}(2N,\bc)$.  So-called {\em infinitesimal} symplectic transformations, corresponding to elements of the Lie algebra $\text{sp}(2N,\bc)$, give rise to vector fields that preserve $\omega$, known as Hamiltonian vector fields.  The $2N^2+N$-dimensional space $\text{sp}(2N,\bc)$ is isomorphic to the $2N^2+N$-dimensional vector space of Hamiltonians spanned by
$$ p_\alpha p_\beta,\quad p_\alpha q_\beta,\quad q_\alpha q_\beta.
$$
$H(p,q)=p_1q_2$ is an example of a Hamiltonian.  Quantisation of these coordinates, i.e. promotion to operators $p_\alpha\mapsto\hat{p}_\alpha$ and $q_\alpha\mapsto\hat{q}_\alpha$, gives an algebra defined by
$$ [\hat{p}_\alpha,\hat{q}_\beta]=\delta_{\alpha\beta}\hbar,\quad[\hat{p}_\alpha,\hat{p}_\beta]=0=[\hat{q}_\alpha,\hat{q}_\beta]
$$
so we naturally choose $\hat{p}_\alpha=\hbar\frac{\partial}{\partial q_\alpha}$ and $\hat{q}_\alpha$ acts by multiplication by $q_\alpha$.  Quantisation of a function, or observable, in $\{p_\alpha,q_\beta\}$ is not unique.  We can consistently define quantisation of the quadratic Hamiltonians by
$$ \widehat{p_\alpha p_\beta}=\hbar^2\frac{\partial^2}{\partial q_\alpha\partial q_\beta},\quad \widehat{p_\alpha q_\beta}=\hbar q_\beta\frac{\partial}{\partial q_\alpha},\quad \widehat{q_\alpha q_\beta}=q_\alpha q_\beta.$$
Linear combinations of these give the quantisation of infinitesimal symplectic transformations.  Hence we can define the quantisation of a linear symplectic transformation $A=\exp{a}$ for $a\in\text{sp}(2N,\bc)$ to be $\hat{A}=\exp{\hat{a}}$.

This construction generalises to the infinite dimensional symplectic manifold $(\ch,\Omega)=(H[z^{-1}][[z]],\Omega)$ for $H\cong\bc^D$, where the symplectic form $ \Omega$ is defined by
$$ \Omega( f(z),g(z))=\res_{z=0}f(-z)g(z)dz.
$$
$\ch_+=H[[z]]$ is Lagrangian with respect to $\Omega$ and $(\ch,\Omega)\cong (T^*\ch_+,\omega^{\text{canonical}})$.  Darboux coordinates for $\Omega$ are $q_{k,\alpha},p_{k,\alpha}$ defined by $\displaystyle\ch\ni f(z)=\sum_{k\geq 0} q_{k,\alpha}z^{k}+\sum_{k< 0} p_{k,\alpha}z^{-k}$.

Given a dimension $D$ vector space $V$ and a sequence of operators $r_m:V\to V$, $m=1,2,...$ such that $r_m(-v)=(-1)^{m+1}r_m(v)$ define
$$\widehat{r_mz^m}:=\sum_{k=0}^{\infty}v^{k,\alpha}(r_m)^\beta_\alpha\frac{\partial}{\partial v^{m+k,\beta}}+\frac{\hbar}{2}\sum_{i=0}^{m-1}(-1)^{i+1}(r_m)^{\alpha,\beta}\frac{\partial^2}{\partial v^{i,\alpha}\partial v^{m-i-1,\beta}}.
$$
Then the quantisation of $R(z)=\exp r(z)$ is given by $\hat{R}(z)=\exp\hat{r}(z)$.

\end{document}